\documentclass[11pt]{amsart}
\usepackage[utf8]{inputenc}
\usepackage[top=1in, bottom=1in, left=1in, right=1in]{geometry}

\usepackage{amsmath,amsfonts,amssymb,stackengine,graphicx,csquotes,mathtools,mathrsfs}
\usepackage{bbold}
\usepackage[normalem]{ulem}

\usepackage{multirow}

\usepackage{hyperref}
\usepackage{color}
\usepackage[dvipsnames]{xcolor}
\usepackage{hyperref}
\hypersetup{
    colorlinks=true, 
    linktoc=all,     
    colorlinks=true,
    linkcolor=MidnightBlue,
    citecolor=blue,
    urlcolor=blue,
		}

\usepackage{booktabs}

\usepackage{graphicx}
\graphicspath{images}
\usepackage{subcaption}
 \usepackage{tikz}
\usetikzlibrary{automata, positioning, arrows}
\tikzset{
        ->,  
        >=stealth, 
        node distance=1.8cm, 
        every state/.style={thick, fill=gray!10}, 
        initial text=$ $, 
        }

\usepackage{amsrefs}

\usepackage[colorinlistoftodos]{todonotes}

\usepackage{amsthm}
\theoremstyle{definition}
\newtheorem{defn}{Definition}[section]

\theoremstyle{plain}
\newtheorem{maintheorem}{Theorem}

\newtheorem{lem}[defn]{Lemma}
\newtheorem{sublem}[defn]{Sublemma}
\newtheorem{thm}{Theorem}
\newtheorem*{thm*}{Theorem}
\newtheorem{prop}[defn]{Proposition}

\theoremstyle{remark}
\newtheorem{rem}[defn]{Remark}

\numberwithin{defn}{section}

\newcommand{\eps}{\varepsilon}

\renewcommand{\phi}{\varphi}
\newcommand{\quand}{\quad\text{ and } \quad}

\newcommand{\dto}[1]{\overset{d}{\longrightarrow}}

\makeatletter
\def\namedlabel#1#2{\begingroup
    #2%
    \def\@currentlabel{#2}%
    \phantomsection\label{#1}\endgroup
}
\makeatother

\begin{document}

\title[Persistent Non-Statistical Dynamics]{Persistent Non-Statistical Dynamics \\ 
in One-Dimensional Maps}

\author{Douglas Coates}
\address{University of Exeter, Exeter, UK}
\email{dc485@exeter.ac.uk}

\author{Stefano Luzzatto}
\address{Abdus Salam International Centre for Theoretical Physics (ICTP), Trieste, Italy}
\email{luzzatto@ictp.it}
\urladdr{www.stefanoluzzatto.net}

\thanks{D.C. was partially supported by the ERC project 692925 \emph{NUHGD} and by the Abdus Salam ICTP visitors program.}
\date{\today}


\date{2 September 2023}


\begin{abstract}
We study a  class \( \widehat{\mathfrak{F}}  \) of one-dimensional full branch maps introduced in \cite{CoaLuzMub22},  admitting two indifferent fixed points as well as critical points and/or singularities with unbounded derivative. We show that  \( \widehat{\mathfrak{F}}  \) can be partitioned into  3 pairwise disjoint subfamilies \( \widehat{\mathfrak{F}}  = \mathfrak{F}  \cup \mathfrak{F}_\pm  \cup \mathfrak{F}_*\)  such that all  \( g \in \mathfrak{F} \) have a unique physical measure equivalent to Lebesgue, all \( g \in \mathfrak{F}_{\pm} \) have a physical measure  which is a Dirac-\(\delta\) measure on one of the (repelling) fixed points, and   all \( g \in \mathfrak{F}_{*} \) are non-statistical and in particular have no physical measure. Moreover we show that  these subfamilies are \emph{intermingled}: they can all be approximated by maps in the other subfamilies in natural topologies. 
\end{abstract}

\maketitle


\section{Introduction}

We study the existence and, especially, the \emph{non-existence} of physical measures in a large family of interval maps  \(   \widehat{\mathfrak{F}} \) which was introduced in \cite{CoaLuzMub22}. We start with heuristic and conceptual overview of results. Then in  Section~\ref{sec:results} we give the formal  definition of the family \(   \widehat{\mathfrak{F}} \) and the precise technical statements of our results.   In Section \ref{sec:recall} we recall the main construction and key estimates from \cite{CoaLuzMub22} which will be required,  and in   Sections \ref{sec:physical}-\ref{sec:density} we prove our results.

\subsection{Overview of Results}\label{sec:overview}

The family   \(   \widehat{\mathfrak{F}} \)  consists of full branch maps with two orientation preserving branches, which are all in the same  topological conjugacy class of uniformly expanding maps such as \( f(x)=2x\)~mod~1, in particular they are all topologically conjugate to each other. Depending on a number of parameters,  they may however exhibit quite different features: the two fixed points are always topologically repelling but may be either hyperbolic or neutral and the branches may have critical points and/or singularities with infinite derivative.   \( \widehat{\mathfrak{F}} \) contains uniformly expanding maps and  well-known intermittent maps as well as many other maps which, as far as we know have not been studied before, see \cite{CoaLuzMub22} for an extensive discussion and review of the literature.   Figure \ref{fig:fig} shows some possible graphs of the maps in \( \widehat{\mathfrak{F}} \), and see Section \ref{sec:defn} for the formal definition.

\begin{figure}[h]\label{fig:1}
 \centering
  \begin{subfigure}{.3\textwidth}
    \centering
    \includegraphics[width=\textwidth]{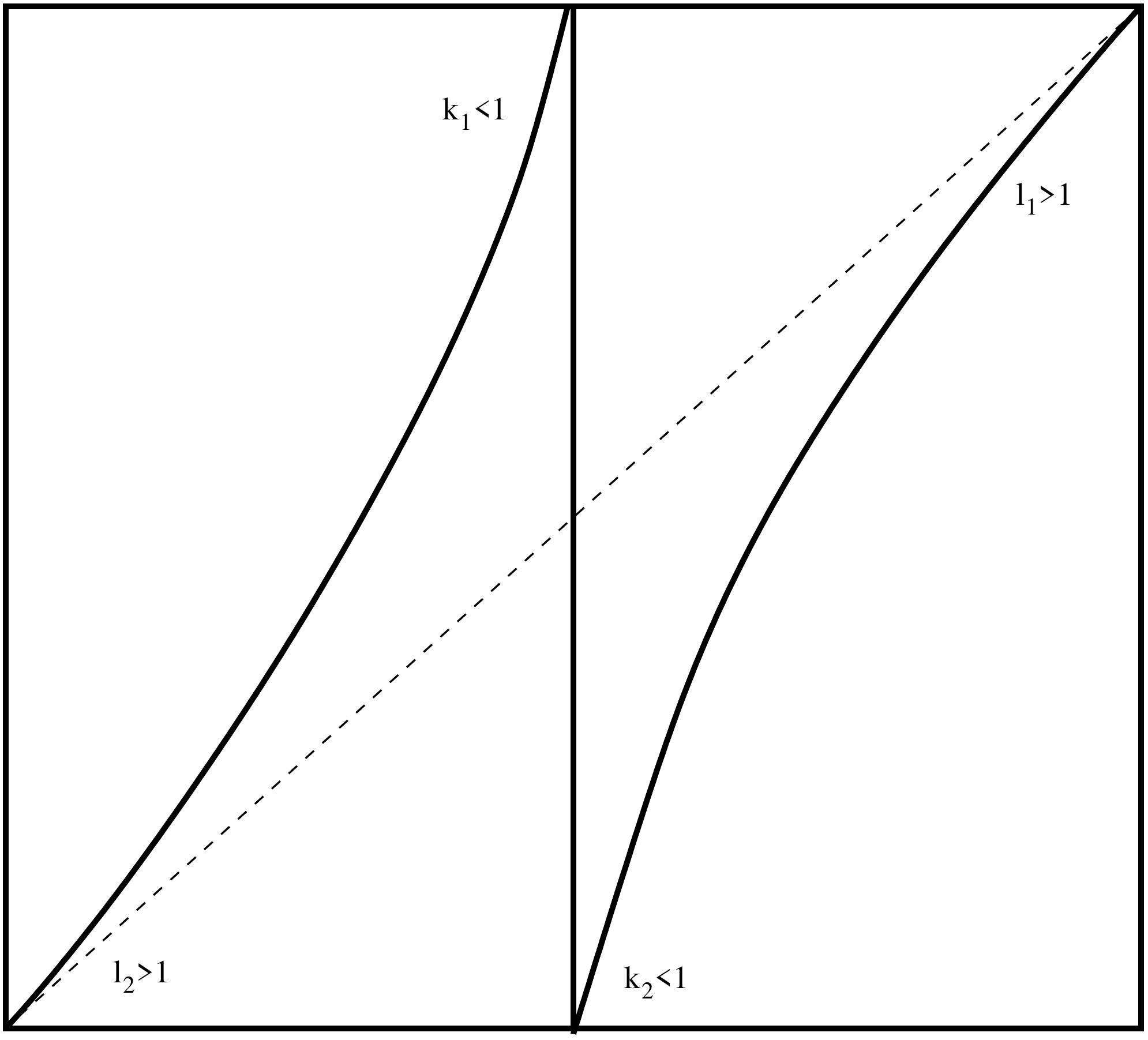}
  \end{subfigure}%
 \quad 
  \begin{subfigure}{.3\textwidth}
    \centering
    \includegraphics[width=\textwidth]{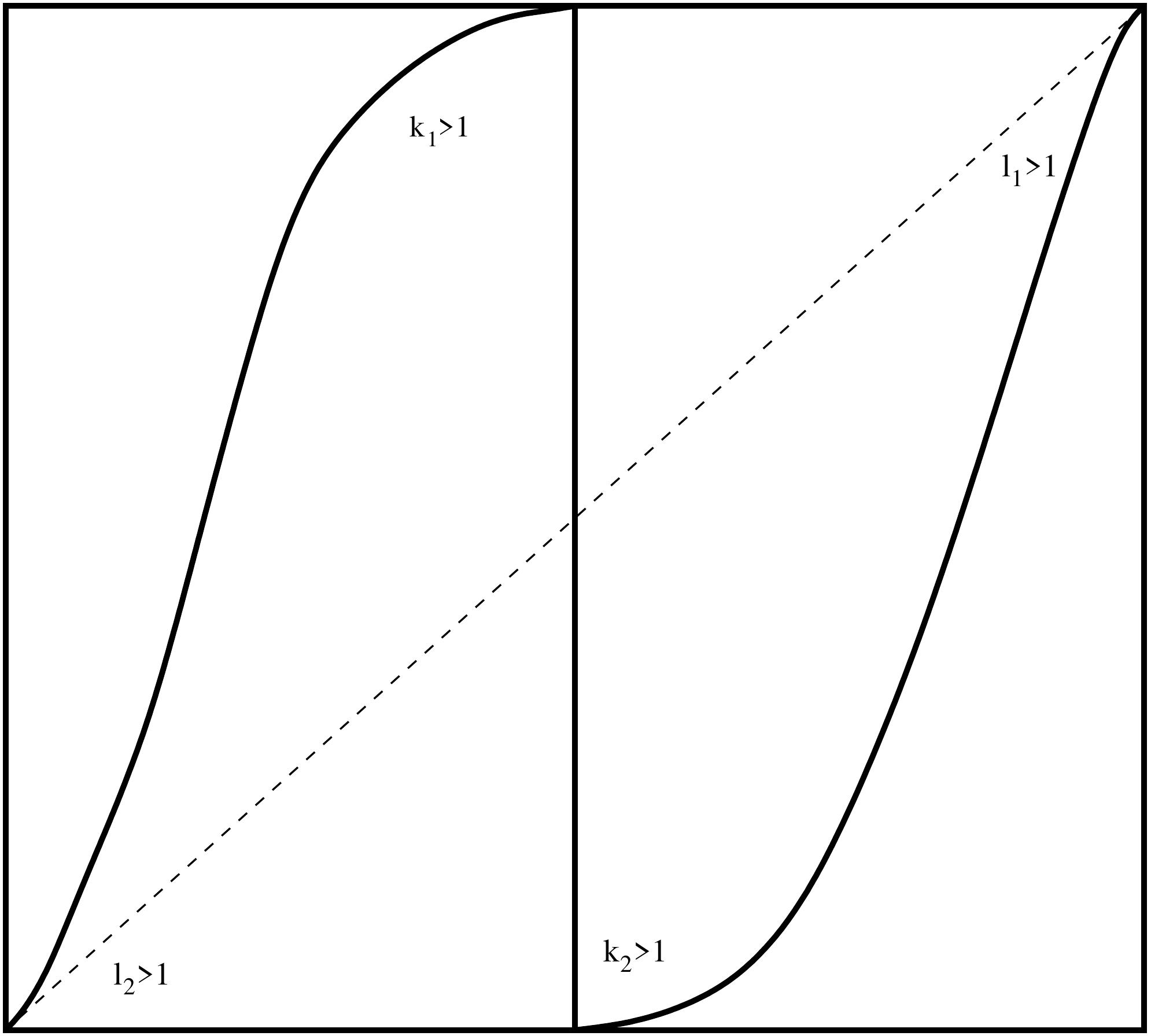}
  \end{subfigure}
\quad
  \begin{subfigure}{.3\textwidth}
    \centering
    \includegraphics[width=\textwidth]{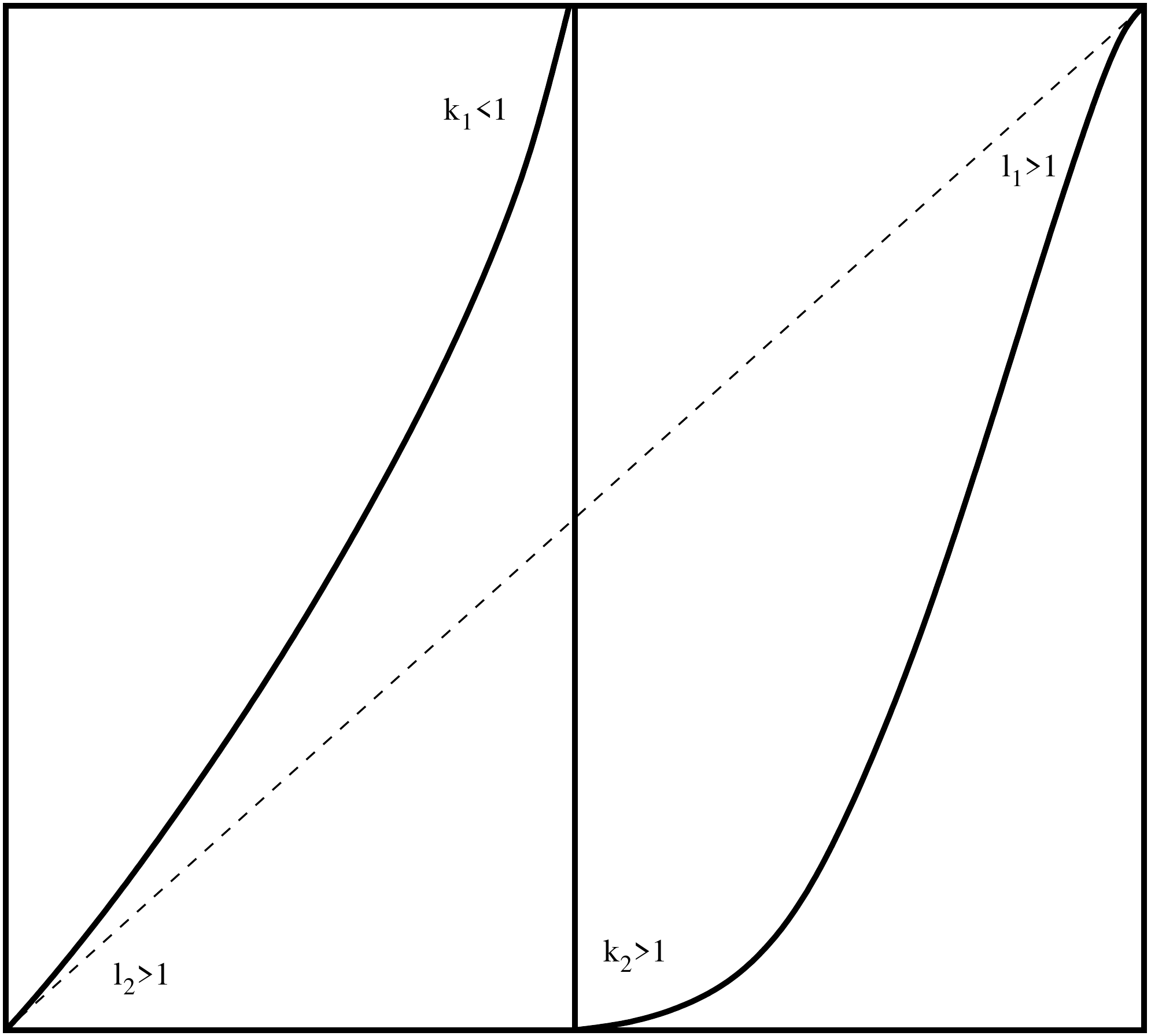}
  \end{subfigure}
  \caption{Graph of $g$ for various possible values of parameters.}\label{fig:fig}
\end{figure}

The first main result of this paper is a complete classification of  maps in \( \widehat{\mathfrak{F}}  \) from the point of view of the kind of physical measure they admit (or not). In particular we will prove the following. 

\begin{maintheorem}\label{mainthmA}
The family of maps  \( \widehat{\mathfrak{F}}  \) is the  union of three non-empty pairwise disjoint subfamilies 
\begin{equation}\label{eq:subfam}
   \widehat{\mathfrak{F}}  = \mathfrak{F}  \cup \mathfrak{F}_\pm  \cup \mathfrak{F}_*
\end{equation}
satisfying the following properties: 

\medskip

1)  all \( g \in \mathfrak{F} \) have a physical measure equivalent to Lebesgue;

2) all \( g \in \mathfrak{F}_{\pm} \) have a physical measure supported on a repelling fixed point;

3)  all \( g \in \mathfrak{F}_{*} \) are non-statistical and in particular have no physical measure.
\end{maintheorem}

The definitions of physical measure and  and non-statistical maps are given in Section \ref{sec:introphsy},  and the families \( \mathfrak{F}, \mathfrak{F}_\pm,  \mathfrak{F}_*\) are  defined explicitly in Section \ref{sec:stat} in terms of various parameters of the maps in \( \widehat{\mathfrak{F}} \) . Item 1) in Theorem \ref{mainthmA}, i.e. the existence of physical measures equivalent to Lebesgue for  maps in  \( \mathfrak{F} \),  was  proved  in \cite{CoaLuzMub22}, where it was also shown that such physical measure satisfy a number of statistical properties such as decay of correlations and various limit theorems  Our focus in this paper is  therefore on the complementary families  \( \mathfrak{F}_{\pm} \) and \(  \mathfrak{F}_{*} \). We note  that the three families \( \mathfrak{F},  \mathfrak{F}_{\pm},  \mathfrak{F}_{*} \) form a partition of \( \widehat{\mathfrak{F}}  \), \emph{there are no cases in which we are not able to obtain a conclusion} (although, as we shall see in the proofs, there are some boundary cases in which more sophisticated arguments are required).

A natural question concerns the  ``\emph{size}'' and ``\emph{structure}'' of the families  \(  \mathfrak{F},  \mathfrak{F}_\pm, \mathfrak{F}_*\) inside \(   \widehat{\mathfrak{F}}  \),  and perhaps one of the most surprising and unexpected results of this paper is that they are \emph{intermingled} and  \emph{dense} in some natural subsets of  \( \widehat{\mathfrak{F}} \) in some appropriate topologies.  Due to the presence of the discontinuity we cannot use  the standard \( C^{r}\), or even \( C^{0}\), metric on  \(   \widehat{\mathfrak{F}}  \), since the maps are not even continuous. However there are pairs of maps  \(  f, g\in  \widehat{\mathfrak{F}}  \) whose \emph{difference} \( f-g\in C^{r}\), and which therefore may be considered as \( C^{r} \) \emph{perturbations} one of the other. This observation motivates the definition of a natural  extended  metric  on \(   \widehat{\mathfrak{F}}  \) defined as follows:
for any \(  f,g \in \widehat{ \mathfrak{F} } \) and \( r \geq 0 \) we let
\begin{equation*}\label{eq:approx}
d_{r}(f,g)\coloneqq \begin{cases} \|f - g \|_{C^r} &\text{ if }   f-g\in C^{r} \\
\infty &\text{ otherwise}. 
\end{cases}
\end{equation*}
For simplicity we will just refer to it as \(C^{r}\) \emph{metric} on  \(   \widehat{\mathfrak{F}}  \)\footnote{Notice that we can ``normalize'' this extended metric to give a standard bounded metric on  \(   \widehat{\mathfrak{F}}  \) by defining 
  \(
    \tilde d_r ( f, g ) \coloneqq  { \tilde d_{r} ( f ,g ) }/{(1 + \tilde d_{r} (f , g ))}
  \)
  when \( d_{r} ( f ,g )< \infty \) and \(  \tilde d_{r} ( f ,g ) \coloneqq 1 \) otherwise.  
The  metrics \( d_{r}\) and \( \tilde d_{r}\) lead to equivalent topologies and so for our purposes it does not really matter which one we use.}.  
There are of course many maps which are at infinite distance from each other but, as we shall see, there is nevertheless a very rich structure in every neighbourhood of maps \( f \in \widehat{ \mathfrak{F} } \) as we have  the following surprising and remarkable result. 

\begin{maintheorem}\label{thm:densityC0}
Each of the families  \(  \mathfrak{F} \),  \( \mathfrak{F}_{\pm}\), \( \mathfrak{F}_*\)  are \(C^0\)-dense in \(\widehat{ \mathfrak{F}}\). 
\end{maintheorem}

Theorem \ref{thm:densityC0} says, more precisely, than any  \( f \in \widehat{ \mathfrak{F} } \), belonging to any  of the classes \(  \mathfrak{F} \),  \( \mathfrak{F}_{\pm}\), \( \mathfrak{F}_*\),  can be approximated arbitrarily closely in the \( C^{0}\) metric by maps in the other two classes. For example, maps with a physical measure equivalent to Lebesgue, including  uniformly expanding maps, can be \(C^0\) approximated both by maps with physical measures given by  Dirac-delta measures on the fixed points and also by maps without physical measures. Similarly, maps  without physical measures can be \(C^0\) approximated both by maps with a physical measure equivalent to Lebesgue and by maps with  Dirac-delta physical measures on the fixed points. 

Theorem \ref{thm:densityC0} will be proved as a special case of a more technical but much more general result, see Theorem~\ref{thm:density-main} below, which also implies that 
maps in the families   \(  \mathfrak{F} \),  \( \mathfrak{F}_{\pm}\), \( \mathfrak{F}_*\) can be approximated by maps in the other families in arbitrarily regular metrics, depending on the maps. 
 In particular,  for the case \( r = 1 \) we define the set 
\[
\widetilde{ \mathfrak{F}}\coloneqq \{ f \in  \widehat{ \mathfrak{F}}: \text{both fixed points are neutral fixed points}\} 
\]
and we have  the following perhaps even more surprising and remarkable result. 

\begin{maintheorem}\label{thm:densityC1}
Each of the families  \(  \mathfrak{F} \),  \( \mathfrak{F}_{\pm}\), \( \mathfrak{F}_*\)  are \(C^1\)-dense in \(\widehat{ \mathfrak{F}}\). 
\end{maintheorem}

Theorem \ref{thm:densityC1} says that every \( C^{1}\) neighbourhood of every map \( f \in  \widetilde{ \mathfrak{F}} \) contains maps belonging to \emph{all three families} \(  \mathfrak{F} \),  \( \mathfrak{F}_{\pm}\), \( \mathfrak{F}_*\). In particular every map with a physical measure equivalent to Lebesgue can be \( C^{1}\) approximated by maps without physical measures, and vice-versa. 

Finally, we address the question of how \emph{persistent} is the dynamics corresponding to the families \(  \mathfrak{F} \),  \( \mathfrak{F}_{\pm}\), \( \mathfrak{F}_*\), i.e. how \emph{``large''} each family is and how \emph{``robust''} with respect to perturbations.  Given Theorems \ref{thm:densityC0} and \ref{thm:densityC1} above, none of these families can be \emph{open} in either the \( C^{0}\) or the \( C^{1} \) metrics as defined above, however we will see that there are several ways in which we can argue that each family is \emph{large} in some sense. In particular, letting \( \operatorname{supp} f \coloneqq \{ x : f(x) \neq 0 \} \) denote the support of function, we will prove the following result.

\begin{maintheorem}\label{mainthm:open}
  Let \( f \in \mathfrak{F}_{*} \). If \( g \in \widetilde{\mathfrak{F}} \) is such that
  \(
    \overline{ \operatorname{ supp } ( f-g ) }  \subset (-1,0)\cup (0,1),
  \)
  then \( g \in \mathfrak{F}_* \). \\
  The same statement is true if we replace \( \mathfrak{ F }_* \) with either of the other two subclasses \( \mathfrak{ F }, \mathfrak{ F }_{\pm} \).
\end{maintheorem}

Theorem \ref{mainthm:open} shows that  the families \(  \mathfrak{F} \),  \( \mathfrak{F}_{\pm}\), \( \mathfrak{F}_*\) are  open under a particular class of perturbations which is slightly more restrictive than those allowed under the general \( C^{r}\) metric defined above but still quite substantial. This is of course particularly remarkable when applied to  maps in \( \mathfrak{F}_{\pm}\) whose physical measures are Dirac-delta measures on fixed points, and  to maps in \( \mathfrak{F}_* \) without physical measures. As far as we know there are no other previously known examples of systems with non-statistical behaviour which is as robust as this. There are also other ways in which the families \(  \mathfrak{F} \),  \( \mathfrak{F}_{\pm}\), \( \mathfrak{F}_*\) are persistent but these need to be formulated in terms of the parameters which define the maps and we therefore postpone the statements to Section \ref{sec:open} below.

\subsection{Physical Measures}\label{sec:introphsy}

In this section we give the precise definition of physical measure  and non-statistical dynamics and give a brief discussion of relevant previously existing results. 

\subsubsection{Topological and Statistical Limits}
For completeness we start with some general background notions which help to motivate the definitions and the results. Given a set \( X \),  a map  \( f : X \to X \) determines a  Dynamical System by  defining the \(n\)'th iterate \( f^{n}=f\circ \cdots \circ f\) by the \(n\)-fold composition of \( f\) with itself, and for any \( x_{0}\in X \) we define the \(n\)'th iterate of \( x_{0} \) under \( f \) as the image \( x_{n}=f^{n}(x)\). We can think of  \( X \) as a \emph{state space}, or the collection of all possible configurations of some system, the map \( f \) as a \emph{force} or \emph{mechanism} which acts on the system, and \( x_{0}\) as an \emph{initial condition}. Then the sequence 
\[
\mathcal O(x) : =\{x_{n}\}_{n=0}^{\infty},
\]
which we call the (forward) \emph{orbit} of \( x_{0}\), denotes the \emph{evolution in time} of the system starting from the initial condition \( x_{0}\). 
The main objective of the Theory of Dynamical Systems is essentially that of describing the structure of  orbits and how they depend on the initial condition \( x_{0}\) and on the map \( f \).   If \( X \) is a \emph{topological space}  we can define the \emph{omega-limit} set
\(
\omega(x):= \{y\in X: y \text{ is an accumulation point of the sequence } \mathcal O(x)\},
\)
which gives information about the asymptotic nature of the orbit from a \emph{topological} point of view. If \( X \) is a \emph{measure space}  we can define the Dirac-delta point mass measures \( \delta_{x_{k}}\) at each point of the orbit and use this to describe the orbit  by equidistributing the mass on the first \( n \) terms of the orbit, giving a sequence 
\[
\mu_{n} (x_{0}) \coloneqq \frac{1}{n} \sum_{ k = 0 }^{ n - 1 }
\delta_{x_{k}}
\]
of  probability measures associated to the initial condition \( x_{0}\). If this sequence \emph{converges}, for example in the weak star-topology, i.e. if there exists a probability measure \( \mu \) such that 
\begin{equation}\label{eq:conv}
 \mu_{n}(x_{0}) \to \mu, 
 \end{equation}
then  \( \mu_{n}\) approximates \( \mu \) but, most importantly,  for all sufficiently  large \( n \geq 0 \),  \( \mu\) \emph{approximates} \( \mu_{n}\),  and therefore \( \mu \) gives an asymptotic description of the orbit from a \emph{statistical} point of view. 

Notice that if \( X \) is a metric space we can describe each orbit \emph{from both a topological and a statistical point of view}. In many cases these two descriptions are intuitively consistent one with the other, for example if \( X \) is a complete metric space and \( f \) is a contraction and \( p \in X \) is the unique fixed point of \( f\), then  is easy to check that for any initial condition \(x_{0}\in X \) the points of the orbit of \( x_{0}\) converge to \( p \) and therefore  \( \omega(x_{0})=\{p\}\) and \( \mu_{n}(x_{0}) \to \delta_{p}\).  Similarly, for irrational circle rotations it is relatively easy to check that the orbit of every point \( x_{0}\in \mathbb S^{1}\) is dense in \( \mathbb S^{1} \), and therefore \( \omega(x_{0})=\mathbb S^{1}\), and it is a classical (but non-trivial) result that every orbit is uniformly distributed and therefore \( \mu_{n}(x) \to \) \emph{Lebesgue}  on \( \mathbb S^{1}\). However this is not always the case and sometimes, such as in several examples which will be discussed below, and   in some of the cases mentioned in our results, the topological and statistical description depend on the initial condition and yield quite different pictures  of what we consider as the ``typical'' dynamics of the system.

\subsubsection{Definition of Physical Measures}
Given a probability measure \( \mu \) we  define its \emph{basin} 
\[
\mathscr{B}_{\mu} \coloneqq \left\{ x : \mu_{n}(x) \to \mu \right\}.
\]
The set \( \mathscr{B}_{\mu} \) may very well be empty but,  if \( \mathscr{B}_{\mu} \neq \emptyset  \) then a natural question is to study it's size. 
Suppose \( X \) is a measure space  with a normalized reference (Lebesgue) measure denoted by \( Leb\). 
\begin{defn}
A probability measure \( \mu \) is a \emph{physical measure} (with  \emph{full measure basin}) for \( f \) if 
 \[
  Leb(\mathscr{B}_{\mu})=1
  \]
\end{defn}
More generally, we say that \( \mu \) is a physical measure if \(   Leb(\mathscr{B}_{\mu}) > 0 \) but the examples we consider in this paper  will always have full measure basin so for simplicity, unless otherwise specified, we will always implicitly assume that physical measures have full measure basins. 

\subsubsection{Physical measures on attractors}
There are plenty of examples of dynamical systems with physical measures, for example: the Dirac-delta \( \delta_{p}\) in the fixed point for a contraction (in which case every point actually belongs to the basin);  Lebesgue measure for irrational rotations (in which case also, every point belongs to the basin) and also for piecewise affine expanding circle maps such as \( f(x) = 2x \) mod 1 (in which case the basin has full Lebesgue measure but its complement is non-empty and indeed consist of an uncountable set of points). It follows by a classical theorem of Birkhoff \cite{Bir31} that if a probability measure \( \mu \) is invariant (\( \mu(f^{-1}(A))=\mu(A)\) for every measurable set \( A \)), ergodic (\( f^{-1}(A)=A \Rightarrow  \mu(A)=0\) or \( \mu(A) =1 \)), and 
\emph{absolutely continuous with respect to Lebesgue} (\(Leb(A)=0 \Rightarrow  \mu(A)=0 \)) then \( \mu \) is a physical measure and   if \( \mu \) is \emph{equivalent to Lebesgue} (\(Leb(A)=0 \Leftrightarrow \mu(A)=0\)), then \( \mu \) is a physical measure with full measure basin.

Such ergodic invariant measures equivalent to, or absolutely continuous with respect to, Lebesgue have been proved to exist in many classes of \emph{uniformly and non-uniformly expanding maps}. There is a huge literature so we just mention a few significant papers such as \cites{AlvLuzPin05,AraLuzVia09,Buz00,DiaHolLuz06,LasYor73,Pin06,Tsu00c,Tsu01a,Tsu05}  but highlight in particular those related to  one-dimensional maps with neutral fixed points \cites{Dol04,BahGalNis18,BahSau16,BalTod16,BruLep13,BruTerTod19,CamIso95,CoaHolTer19,CriHayMar10,Cui21,Dua12,FisLop01,FreFreTod13,FreFreTod16,FroMurSta11,Ino00,Kor16,LivSauVai99,Mel08,MelTer12,NicTorVai16,Pia80,PolSha09,PolWei99,PomMan80,Ruz15,Ruz18,Sar01,SheStr13,Ter13,Ter15,Tha00,Tha05,Tha80,Tha83,Tha95,Tha95a,You99,Zwe00,Zwe03,Zwe98}.

 In higher dimensions, where we may have a non-trivial attractor \( \Lambda\) of zero Lebesgue measure,  if \( \Lambda\) satisfies some hyperbolicity conditions then  pioneering work of Anosov, Sinai, Ruelle, and Bowen in the 1960s and 1970s showed that 
 the absolute continuity  can be replaced by a weaker but more technical condition of \emph{absolute continuity of conditional measures} and \emph{absolute continuity of the stable foliation}, and measures satisfying such properties are often referred to as \emph{Sinai-Ruelle-Bowen}, or \emph{SRB}, measures \cites{AnoSin67,Bow75, Rue76, Sin72}. 
  Over the following decades  then there has been a tremendous amount of research which has extended their results to increasingly general classes of systems, we mention here just some of the more recent papers \cites{AlvDiaLuz17,Bur21,BuzCroSar22,CliLuzPes17, CliLuzPes23,Vec22} and refer the reader to those for a more comprehensive list of references,  and the formulation of a far-reaching conjecture of Palis to the effect that ``typical'' dynamical systems have (a finite number of) physical measures \cites{Pal08, Pal15}.

\subsubsection{Physical measures on repellors}
A rarer, but in many ways more interesting and intriguing, class of examples of physical measures consists of systems in which the physical measure is supported on an invariant set \( \Lambda\) which has zero Lebesgue measure and is topologically \emph{repelling}, in the sense that points close to \( \Lambda\) are mapped \emph{away} from \( \Lambda\) rather than \emph{towards} \( \Lambda\). The simplest example of this phenomenon, which is also very relevant for the class of maps which we consider in this paper, is given by well known  Manneville-Pomeau \emph{intermittency map} \( f(x)=x+x^2\) mod 1 \cite{PomMan80}.  It is easy to see that \( f(0)=0 \), and so the origin is a fixed point, and that \( f'(x) = 1+ 2x\), so that \( f'(0)=1\) and \(f'(x)>1\) for all \( x> 0 \). In particular every point in any small neighbourhood of the origin, except the origin itself, eventually leaves such a neighbourhood, and in this sense the origin is topologically repelling. However it can be shown that  asymptotically most points  of the orbits belong to arbitrarily small neighbourhoods of the origin, and in fact \emph{the empirical measures \( \mu_n(x)\)  converge to the Dirac-delta measure \( \delta_0\) at the origin}, which is therefore a physical measure with full basin, see \cite{Alv20} and some of the paper mentioned above. 
Notice that the Manneville-Pomeau intermittency map is included in  our family  \( \widehat{\mathfrak F}\), see discussion in \cite{CoaLuzMub22}, and the result just mentioned is a special case of our Theorem \ref{thm:phys-measures} below. 
In the 1990s, Hofbauer and Keller proved  the even more surprising result that there are many examples in smooth quadratic unimodal maps for which a similar phenomena occurs \cites{HofKel90,  HofKel95}.

\subsubsection{Non-Statistical Maps}\label{sec:nonstat}

Equally, if not even more, interesting are dynamical systems which  \emph{do not admit any physical measure.} The simplest, and somewhat trivial, way in which this can occur is when the basin of every probability measure \( \mu \) has zero Lebesgue measure, such as in the  identity map for which \( \mu_{n}(x) \to \delta_{x}\), for every \( x \) or for rational circle rotations for which all orbits are periodic.  A much more sophisticated and interesting way in which a map can fail to have physical measures is when there exists a full measure set of points for which the sequence of measures \( \mu_{n}(x)\) \emph{does not converge}, in this case we say that the orbit of \( x \) is \emph{non-statistical}. More formally, letting 
\[
\mathscr N:=\{x\in X: \mu_n(x) \text{ does not converge} \}
\]
we can make the following definition. 
\begin{defn}
A map \( f: X \to X \)  is  \emph{non-statistical}  if 
 \[
  Leb(\mathscr{N})=1
  \]
\end{defn}

Notice that non-convergence of  \( \mu_{n}(x)\)  means that there must exist at least two measures, \( \hat\mu, \tilde\mu\), and two subsequences \( n_{i}\to \infty, n_{j}\to \infty\) such that \( \mu_{n_{i}}(x) \to \hat \mu\)  and \( \mu_{n_{j}}(x) \to \tilde \mu\). This   means that there is an infinite sequence of times  for which the statistics of the orbit is extremely well described by the measure \( \hat \mu\) and another sequence  of times  for which the statistics of the orbit is extremely well described by the measure \( \tilde \mu\). We can think of these sequences as defining a series of \emph{timescales} at which we see completely different statistical behaviour and therefore the observed frequency of visits to any particular region does not stabilize as \( n \to \infty\). 

There is  quite a large bibliography of research exploring the notion of the non-existence of physical measures from different points of view and giving  a number, albeit quite limited, of examples \cites{AarThaZwe05,AraPin21,Bac99,BarKirNak20,BerBie22,ColVar01,CroYanZha20,Her18,HofKel90,HofKel95,HuYou95,Ino00,JarTol04,KanKirLi16,KarAsh11,Kel04,KirLiNak22,KirLiSom10,KirNakSom19,KirNakSom21,KirNakSom22,KirSom17,Kle06,LabRod17,Pal15,Tak08,Tak94,Tal20,Tal22,Zwe02}.  We give here only a short and non comprehensive review of some of these and  refer the reader to the original papers  for additional information. 

Arguably the first example of a non-statistical system is the \emph{Bowen eye}, attributed by Takens~\cite{Tak94} to Bowen in the early 70s. The Bowen eye is a two dimensional vector field with an eye-like region whose boundary is formed by two saddle connections between two fixed points and under carefully chosen conditions orbits tend to oscillate between the two in a non-statistical way. 

This example is somewhat ``mild'' because the dynamics is very simple but,  
around the same time Hofbauer and Keller \cites{HofKel90, HofKel95, Kel04}  showed that there are (uncountably) many parameters in the logistic family \( f_\lambda (x) \coloneqq \lambda x ( 1 - x ) \) for which \( f_{\lambda} \) is \emph{topologically mixing} but \emph{non-statistical}. Very recently  Talebi \cite{Tal22} generalized and extended this result to the setting of complex rational maps. 
Another approach to the construction of non-statistical examples, or at least examples with positive Lebesgue measure of non-statistical points, is by constructing  \emph{wandering domains} with non-statistical dynamics \cites{ColVar01, KirSom17,BerBie22} and a further  example of non-statistical behaviour appears in \cite{CroYanZha20} where the authors construct a skew product \( F : \mathbb{T}^2 \times \mathbb{R} \to \mathbb{T}^2 \times \mathbb{R} \)  which gives rise to non-statistical behaviour using  the fact that skew translations over Anosov diffeomorphisms share properties with Brownian motion.
Some results related to ours, for interval maps with two neutral fixed points, were also obtained in \cites{Zwe00, AarThaZwe05} in a somewhat more abstract setting and with a particular focus on the existence and properties of a sigma-finite  invariant measure. 

As far as we know, none of the existing results considers  a class of maps anywhere near as large as  the family \(  \widehat{\mathfrak{F}} \) considered here, nor gives such a complete and a systematic characterization of the various kinds of physical measures as given in Theorem \ref{mainthmA}.  Most importantly, none of the existing results comes anywhere close to constructing examples of toplogically mixing maps without physical measures which are so \emph{prevalent} and \emph{persistent}, as described in Theorems \ref{thm:densityC0}, \ref{thm:densityC1}, \ref{mainthm:open}.

\section{Statement of Results}\label{sec:results}

We now give the precise definition of the family of maps  \(\widehat{ \mathfrak{F}}\) and the subfamilies \(  \mathfrak{F} \),  \( \mathfrak{F}_{\pm}\), \( \mathfrak{F}_*\) as well as some more general technical theorems which imply the main theorems  in Section \ref{sec:overview} above.

\subsection{Doubly Intermittent Full Branch Maps}\label{sec:defn}

We consider the class of maps introduced in~\cite{CoaLuzMub22}. For completeness we recall the precise definitions. 
  Let \( I, I_-, I_+\) be compact intervals, let  \( \mathring I, \mathring I_-, \mathring I_+\) denote their interiors, and suppose that  \(I = I_{-}\cup I_{+}  \) and \( \mathring I_-\cap \mathring I_+=\emptyset\).

\begin{description}
\item{\namedlabel{itm:A0}{\textbf{(A0)}}}
  \( g: I \to I \) is \emph{full branch}: the restrictions \(g_{-}: \mathring I_{-}\to \mathring I\) and \(g_{+}: \mathring I_{+}\to \mathring I\) are  orientation preserving \( C^{2} \) diffeomorphisms and the only  fixed points are the endpoints of \( I \).
\end{description}

To simplify the notation we assume that
\( I=[-1, 1], 	I_{-}=[-1, 0], I_{+}=[0,1]\)
but our results  will be easily seen to hold in the general setting.
For \( \iota > 0 \), we let 
\( 
  U_{0-}:=(-\iota, 0],
  U_{0+}:=[0, \iota),
  U_{-1}:=g(U_{0+}),
  U_{+1}:=g(U_{0-})
\) 
be one-sided neighbourhoods of the endpoint of the intervals \(I = I_{-},  I_{+}  \).  
\begin{description}
  \item{\namedlabel{itm:A1}{\textbf{(A1)}}}
  There exists constants \( \ell_1,\ell_2 \geq 0 \), \( \iota,  k_1,k_2 , a_1,a_2,b_1,b_2 > 0 \)  such that, if \(  \ell_1,\ell_2> 0 \), \(  k_1,k_2 \neq 1\),
  \begin{equation}\label{eqn_1}
    g(x) =
    \begin{cases}
      x+b_1{(1+x)}^{1+\ell_1} & \text{in }  U_{-1}, \\
      1-a_1{|x|}^{k_1}        & \text{in } U_{0-},  \\
      -1+a_2{x}^{k_2}         & \text{in }  U_{0+}, \\
      x-b_2{(1-x)}^{1+\ell_2} & \text{in }  U_{+1},
    \end{cases}
  \end{equation}
  If  \( \ell_{1}=0 \)  and/or \( \ell_2=0\) we replace the corresponding lines in~\eqref{eqn_1} with
  \(
  g|_{U_{\pm 1}}(x) \coloneqq \pm 1 + (1 + b_1) ( x + 1) \mp  \eta (x),
  \)
  where \( \eta \) is \( C^2\),  \(\eta(\pm 1)= 0, \eta'(\pm 1)=0\), and \( \eta''(x)>0\) on \( U_{-1}\) and  \( \eta''(x)<0\) on \( U_{+1} \).
  If \( k_1 = 1 \) and/or \( k_2 = 1 \), then we replace the corresponding lines in~\eqref{eqn_1} with the assumption that
  \(
  g'(0_-) = a_1>1\) and/or \( g'(0_+) = a_2>1
  \)
  respectively, and that \( g \) is  monotone in the corresponding neighbourhood, which makes the definition much less restrictive.
\end{description}

It is easy to see that the definition in \eqref{eqn_1} yields maps with dramatically different derivative behaviour depending on the values of \( \ell_1, \ell_2, k_1, k_2\), including having neutral or expanding fixed points and points with zero or infinite derivative.

Our final assumption can be \emph{intuitively thought of as saying  that \( g \) is uniformly expanding outside the neighbourhoods \( U_{0\pm}\) and \( U_{\pm 1}\)}. This is however much stronger than what is needed and therefore we formulate a weaker and more general assumption for which we need to describe some aspects   of the topological structure of maps satisfying condition \ref{itm:A0}. First of all we  define
\begin{equation}\label{eq:Delta0}
  \Delta^-_0:=
  g^{-1}(0,1)\cap I_-
  \quand
  \Delta^+_0:=
  g^{-1}(-1,0)\cap I_+.
\end{equation}
Then we define iteratively, for every \( n \geq 1 \),  the sets
\begin{equation}\label{eq:Delta}
  \Delta_n^{-}:= g^{-1}(\Delta_{n-1}^{-})\cap I_{-}
  \quand
  \Delta_n^{+}:= g^{-1}(\Delta_{n-1}^{+})\cap I_{+}
\end{equation}
as the \( n\)'th preimages of \( \Delta_0^-, \Delta_0^+\) inside the intervals \(I_{-}, I_{+} \).   It follows from \ref{itm:A0} that
\(
\{ \Delta_n^{-}\}_{n\geq 0}
\)
and \(
\{ \Delta_n^{+}\}_{n\geq 0}
\)
are $\bmod\;0$ partitions of \(I_{-}\) and \(I_{+}\) respectively, and that the partition elements depend \emph{monotonically} on the index  in the sense that \( n > m \) implies that \( \Delta_n^{\pm}\) is closer to \( \pm 1\) than \( \Delta_m^{\pm}\),  in particular the only accumulation points of these partitions are \( -1\) and \( 1 \) respectively.
Then, for every \( n \geq 1 \),  we let
\begin{equation}\label{eq:delta}
	\delta_{n}^{-}:=
	g^{-1}(\Delta_{n-1}^{+}) \cap  \Delta_0^{-}
	\quand
	\delta_{n}^{+}:=
	g^{-1}(\Delta_{n-1}^{-}) \cap  \Delta_0^{+}.
\end{equation}
Notice  that
\(
\{ \delta_n^{-}\}_{n\geq 1}
\)
and
\(
\{ \delta_n^{+}\}_{n\geq 1}
\)
are $\bmod\; 0$ partitions of \( \Delta_0^-\) and \( \Delta_0^+\)  respectively and also in these cases the partition elements depend monotonically on the index in the sense that \( n > m \) implies that \( \delta_n^{\pm}\) is closer to \( 0 \) than \( \delta_m^{\pm}\),   (and in particular the only accumulation point of these partitions is 0). Notice moreover, that
\(
g^{n}(\delta_{n}^{-})= \Delta_{0}^{+} \) and \(  g^{n}(\delta_{n}^{+})= \Delta_{0}^{-}.
\)
We now define two non-negative integers \( n_{\pm}\) which depend on the positions of the partition elements \( \delta_{n}^{\pm}\) and on the sizes of the neighbourhoods \( U_{0\pm}\) on which the map \( g \) is explicitly defined.
If  \( \Delta_0^{-} \subseteq U_{0-}\) and/or  \( \Delta_0^{+} \subseteq U_{0+}\), we define   \( n_{-}= 0  \) and/or \( n_{+}=0 \) respectively, otherwise we let
\begin{equation}\label{eq:n+-}
  n_{+} := \min \{n :\delta_{n}^{+} \subset U_{0+} \} \quand
  n_{-} := \min \{n :\delta_{n}^{-} \subset U_{0-} \}.
\end{equation}
We can now formulate our final assumption as follows.
\begin{description}
	\item[\namedlabel{itm:A2}{\textbf{(A2)}}]
		There exists a \( \lambda >  1 \) such that  for all   \( 1\leq n\leq n_{\pm}\)   and for all \(    x \in \delta_n^{\pm} \)  we have \(  (g^n)'(x) > \lambda\).
\end{description}
Following \cite{CoaLuzMub22},  we  let
\[
	\widehat{\mathfrak{F}} \coloneqq \{ g : I \to I  \text{ which satisfy \ref{itm:A0}-\ref{itm:A2}}\}
\]
The class \( \widehat{\mathfrak{F}} \) contains  many  maps which have been studied in the literature, including uniformly expanding maps and various well known intermittency maps with a single neutral fixed point, we refer the reader to  \cite{CoaLuzMub22}  for a detailed literature review. 

 \subsection{Physical measures on repelling fixed points  and non-statistical dynamics}\label{sec:stat}
  It is proved in  \cite{CoaLuzMub22}  that every \( g \in \widehat{\mathfrak{F}} \) admits a unique (up to scaling) \( \sigma\)-\emph{finite ergodic invariant measure} \( \hat \mu\) \emph{equivalent to Lebesgue} and that many properties depend on the constants
 \[
 \beta^- := \ell_1 k _2,  \quad  \beta^+ := \ell_2 k_1, \quand
   \beta := \max \{ \beta^+, \beta^- \}.
   \]
 Notice that \(  \beta^-, \beta^+\in [0, \infty) \) and can take any value in the allowed range, depending on the values of \( \ell_1, \ell_2, k_1, k_2\).  They determine the level of \emph{``stickiness''} of the  fixed points \( -1\) and \(+1\) respectively, given by the combination of the  constants \( \ell_1, \ell_2\), which determine the  order of tangency of the graph of \( g \) with the diagonal, and the the constants \( k_1, k_2\), which give the order of the singular or critical points. The larger the value of  \(  \beta^-, \beta^+  \) the \emph{``stickier''} are the corresponding fixed points. We can now define explicitly the subfamilies in \eqref{eq:subfam} by letting 
\begin{equation}
  \label{eq:def-subclasses}
  \mathfrak{F} \coloneqq \{\beta \in [0,1)\} ,
  \qquad 
  \mathfrak{F}_\pm \coloneqq 
  \{\beta\geq 1, \ \beta^- \neq  \beta^+ \}
  \qquad
  \mathfrak{F}_* \coloneqq 
  \{\beta\geq 1, \ \beta^- = \beta^+ \}.	
\end{equation}
    It is clear that these families are pairwise disjoint and that their union is exactly \(   \widehat{\mathfrak{F}}  \). 
    Notice that
    \( \beta \in [0,1) \) implies that \emph{both} \( \beta^{-}, \beta^{+} \in [0,1) \), whereas  \( \beta\geq 1 \) only implies that  \emph{at least one} of \( \beta^-,  \beta^+ \geq 1 \). 
 It is proved in  \cite{CoaLuzMub22} that the \( \sigma\)-finite invariant measure \( \hat\mu\) is   finite, and can therefore be rescaled to a probability measure \( \mu\),  \emph{if and only if} \( \beta\in [0,1)\).
   
 \begin{thm}[\cite{CoaLuzMub22}]\label{thm:CoaLuzMub22}
If \( g \in \mathfrak{F} \) then \( g \) admits a  physical measure equivalent to Lebesgue.
 \end{thm}
 
 As mentioned above, this proves  1) in Theorem  \ref{mainthmA}. 
 We are therefore interested  in the families  \( \mathfrak{F}_\pm  \) and \( \mathfrak{F}_*  \), neither of which can contain any map with a physical measure equivalent to Lebesgue. The maps in \( \mathfrak{F}_\pm \)  are those where one fixed point is \emph{stickier} than the other, whereas the maps in \( \mathfrak{F}_*\) are those for which the stickiness is the same, at least as far as it can be measured by the constants \( \beta^-, \beta^+\).  It turns out that  the typical statistical behaviour  is completely different in these two cases. 

\begin{thm}\label{thm:phys-measures}
If \( g\in \mathfrak{F}_\pm \) then \( g \) admits a physical measure with full basin. Moreover: 
\\
1) if \( \beta^{-}> \beta^{+}\), the physical measure is the Dirac-delta measure \( \delta_{-1}\) on the fixed point -1; 
\\
2) if \( \beta^{+}> \beta^{-}\), the physical measure  is the Dirac-delta measure \( \delta_{1}\) on the fixed point +1 .  
\end{thm}

When \( \beta^{+}=\beta^{-}\), the two neutral fixed points are, in some sense, ``equally sticky'' and it is natural to conjecture that typical orbits would oscillate  between the two, spending positive proportion of time near each one. There are, however, several quite different ways in which this can happen. Let
\begin{equation}\label{eq:Omega}
\Omega \coloneqq \{ \nu_{p} \coloneqq p\delta_1 + ( 1 - p ) \delta_{-1} : p \in [{0,1}]\} 
\end{equation}
be the space of all convex combinations of the Dirac-delta measures \( \delta_{-1}\) and \( \delta_{1}\).  We then have the following quite remarkable statement for maps in \( \mathfrak{F}_* \).

\begin{thm}\label{thm:no-phys-measures}
 If  \( g\in \mathfrak{F}_* \), then for Lebesgue almost every   \( x\in I\), the sequence 
\( 
\mu_n(x) \)
does not converge and so  \( g \) is  non-statistical.  Moreover,  the set of accumulation points of the sequence 
\( 
\mu_n(x) \) is  \( \Omega \). 
\end{thm}

Theorems \ref{thm:CoaLuzMub22}, \ref{thm:phys-measures}, and \ref{thm:no-phys-measures}, clearly imply Theorem  \ref{mainthmA}. For \( g\in \mathfrak{F}_* \), Theorem \ref{thm:no-phys-measures} also gives additional and non-trivial information about the limit points of the sequence \( \mu_{n}(x)\). 

\begin{rem}
It is interesting to  note that, contrary to what might be expected,  there are no cases in which there are  \emph{physical} measures which are non-trivial convex combinations 
\( \nu_{p}\) with \( p \in (0,1)\) of the Dirac-delta measures.  One may expect that this may be achieved at least for some carefully chosen values of the multiplicative paramaters \( a_1, a_2, b_1. b_2\) in \eqref{eqn_1} but in fact our results show that these play no significant role, at least at this level of the description of the dynamics. 
\end{rem}

\subsection{Density of \(  \mathfrak{F},  \mathfrak{F}_\pm, \mathfrak{F}_*\)  in \(   \widehat{\mathfrak{F}}  \)}
\label{sec:dense}
We now address the issue of the density of the families \(  \mathfrak{F},  \mathfrak{F}_\pm, \mathfrak{F}_*\), as stated in Theorems \ref{thm:densityC0} and~\ref{thm:densityC1}. 
We will actually state a much more general  result which says that each map \( f\in  \widehat{\mathfrak{F}}  \)  can be approximated arbitrarily closely in the \( C^{r}\) topology  by maps in \emph{any} of the families \(  \mathfrak{F},  \mathfrak{F}_\pm, \mathfrak{F}_*\),  for some \( r \) which depends both on \( f \) and on the family in which we want to approximate \( f \). We recall first of all the \emph{ceiling function} 
 \( \lceil x \rceil :=min\{\kappa\in \mathbb N: x \leq \kappa\}.
 \)
Then,  for every \( f \in \widehat{\mathfrak{F}} \), we  define 
\[
  r_{\pm} \coloneqq  r_{\pm}(f) \coloneqq \max \{ \lceil \ell_1 \rceil, \lceil \ell_2 \rceil \}, 
  \qquad 
\tilde r \coloneqq \tilde r(f) \coloneqq
  \begin{cases}
    \lceil 1 / k_{2}\rceil & \text{if } 0 \leq \beta^+ < 1 \leq \beta^- \\
    \lceil 1 / k_{1}\rceil & \text{if } 0 \leq \beta^- < 1 \leq \beta^+ \\
    \min\{ \lceil 1 / k_{2}\rceil \lceil 1 / k_{1}\rceil \} &\text{otherwise, }
  \end{cases}  
  \]
  and 
\[r_* \coloneqq r_*(f) \coloneqq
  \begin{cases}
    \min \{ \lceil \ell_1 \rceil, \lceil \ell_2 \rceil \} & \text{if } \beta \in [0,1) \\
    \lceil \ell_2 \rceil &\text{if } \beta^+ < 1 \leq \beta^- \\
    \lceil \ell_1 \rceil &\text{if } \beta^- < 1 \leq \beta^+ \\
    \min \left\{ \lceil \ell_{2} k_{1} / k_2 \rceil, \lceil \ell_1 \rceil \right\} & \text{if } \beta^+, \beta^- \geq 1 \text{ and } k_2 \geq k_1 \\
    \min \left\{ \lceil \ell_{1}k_{2} / k_1  \rceil, \lceil \ell_2 \rceil \right\} & \text{if } \beta^+, \beta^- > 1 \text{ and } k_1 \geq k_2. 
  \end{cases}
\]

\medskip\noindent
Notice that  \( r_{\pm}, r_{*}, r   \) are all well defined non-negative \emph{integers} because of the way there are defined using the ceiling function. Moreover, \( r_{*}=0\) if and only if at least one of the fixed points is hyperbolic, and \(  r_{\pm} = 0 \) if and only if both fixed points are hyperbolic (e.g. if \( f \) is uniformly expanding). If both fixed points are neutral then we necessarily have \( r_{\pm}, r_{*}, r >0  \) and therefore \( r_{\pm}, r_{*}, r \geq 1  \),  since they are all integers and defined in terms of ceiling functions. 

\begin{thm}
  \label{thm:density-main}
  For every \( f \in \widehat{\mathfrak{F}} \) and every \( \varepsilon > 0 \) there exists \( \tilde{f} \in \mathfrak{F} \), \( f_{\pm} \in \mathfrak{F}_{\pm} \) and \( f_{*} \in \mathfrak{F}_{*} \) such that
  \begin{equation*}
    \label{eq:f-g-c-r-close}
    d_{\tilde r} (f , \tilde f) < \varepsilon, \qquad d_{ r_{\pm} } ( f, f_{\pm} ) < \varepsilon, \qquad
    d_{ r_* } ( f, f_* ) < \varepsilon.
  \end{equation*}
\end{thm}

\medskip
Theorem \ref{thm:density-main} immediately  implies Theorems \ref{thm:densityC0} and~\ref{thm:densityC1} since we always have  \( r_{\pm}, r_{*}, \tilde r \geq 0  \) and, as mentioned in the previous paragraph, if \( f \in \widetilde{\mathfrak{F}} \)  (i.e. if both fixed points are neutral) then we always have \( r_{\pm}, r_{*}, \tilde r \geq 1  \). However, it also shows that in some cases we can have approximations in much higher topologies. For example, consider a map \( f\in \mathfrak{F}_{*} \), which therefore has no physical measure. By definition we have \( \ell_{1}k_{2}= \ell_{2}k_{1} = \beta\),  where \( \beta \geq 1 \) is arbitrary. For definiteness let us suppose that \( \beta =1 \) and then, given any \emph{arbitrarily large} positive integer \( R \), there exists a map \( f\in \mathfrak{F}_{*} \) such that
\( \ell_{1} = \ell_{2} = R \) and  \( k_{1} =  k_{2} = 1/R\). This implies \( \tilde r=r_{\pm}=R\) and therefore,  from Theorem \ref{thm:density-main},  we get that the map \( f \), which does not have any physical measure, can be approximated arbitrarily closely in the \( C^{R}\) topology by maps in \( \mathfrak{F} \), which have a physical measure equivalent to Lebesgue, \emph{and} by maps in   \( \mathfrak{F}_{\pm} \), which have physical measures which are Dirac-delta measures on a fixed point. Notice that we do not need to consider \( r_{*}\) since taking  \( g = f \)  the last approximation is trivial.

\subsection{``Openness'' of \(  \mathfrak{F},  \mathfrak{F}_\pm, \mathfrak{F}_*\)  in \(   \widehat{\mathfrak{F}}  \)}\label{sec:open}

Finally, it just remains to discuss the ``openness'' of  the families \(  \mathfrak{F},  \mathfrak{F}_\pm, \mathfrak{F}_*\) as described in Theorem \ref{mainthm:open}. 
Now that we have the formal definitions of the maps in  \(   \widehat{\mathfrak{F}}  \) 
the statement in Theorem \ref{mainthm:open} is actually almost immediate and therefore we just give the proof.

\begin{proof}[Proof of Theorem \ref{mainthm:open}]
By assumption the map \( g \) is in the class \( \widehat{ \mathfrak{ F } } \) and since \( \overline{ \operatorname{ supp } f - g } \subset ( -1 , 0  ) \cup ( 0, 1 ) \) we necesarrily have that \( g \) satisfies \ref{itm:A1} with the same parameters as \( f \). Thus, \( \beta^+(g) = \beta^+(f) \), \( \beta^-(g) = \beta^- (f) \) and so \( g \) must lies in the same subclass \( \mathfrak{F}, \mathfrak{F}_{\pm}, \mathfrak{F}_* \) as \( f \).
\end{proof}

We also mention, without giving formal statements, a couple of other natural ways in which maps in \(  \mathfrak{F}, \mathfrak{F}_\pm,   \mathfrak{F}_*\)  can be perturbed without falling outside of their original family, essentially by perturbing some of the parameters through which they are defined. 
Notice first of all that  the conditions which determine whether \( g \) belongs to  \(  \mathfrak{F} \),  \( \mathfrak{F}_{\pm}\), or \( \mathfrak{F}_*\), do not depend on  the  constants \( a_{1}, a_{2}, b_{1}, b_{2}\) and therefore we can choose these arbitrarily without changing the values of \( \beta_{1}, \beta_{2}\).  Sufficiently small perturbations of these parameters do not invalidate condition (A2) and therefore  each subfamily  \(  \mathfrak{F} \),  \( \mathfrak{F}_{\pm}\), \( \mathfrak{F}_*\) is  also ``\emph{open}'' in the sense that  there exists an open neighbourhood of the parameters \( a_{1}, a_{2}, b_{1}, b_{2}\) for which  the corresponding maps still belong to the same subfamily.

 In addition to the perturbations mentioned above, which do not change the values of the parameters \( \beta^-, \beta^+\), we can also perturb the parameters \( \ell_1, \ell_2, k_1, k_2\) which make up \( \beta^-, \beta^+\). This may of course affect which subfamily the perturbed map belongs to as the subfamilies are precisely defined in terms of the values of \( \beta^-, \beta^+\), and indeed for these kinds of perturbations the situation is slightly different depending on which subfamily we consider. 
 The maps in  \(  \mathfrak{F} \) are characterized by the property that \( \beta^-, \beta^+ \in [0,1)\) and therefore there is an open set of sufficiently small perturbations  of \( \ell_1, \ell_2, k_1, k_2\) such that this still holds as well as condition (A2), thus guaranteeing the the perturbed map is still in \(  \mathfrak{F} \). Similarly, maps in \( \mathfrak{F}_{\pm}\) are characterized by the property that at least one of \( \beta^-, \beta^+\) is \( \geq 1 \) and so again there is a large set of perturbations of \( \ell_1, \ell_2, k_1, k_2\) which preserve that condition. Notice however that this may not always contain an open set of parameters, for example  if \( \beta^-<1\) and \( \beta^+=1\), in which case we can only perturb \( \ell_2\) and \( k_1\) in such a way that \( \beta^+ \) does not decrease. Finally, maps in \( \mathfrak{F}_*\) are defined in the most restrictive way since they require \( \beta^-= \beta^+ \) and this condition is not preserved for an open set of choices of parameters \( \ell_1, \ell_2, k_1, k_2\). 
 Nevertheless it is still a relatively large and persistent family since we can define a 
 
 \begin{quote}
 \begin{center}
 \emph{three-parameter family completely contained in \( \mathfrak{F}^* \)}: 
\end{center}
  \end{quote}
 for \emph{any} \( \beta\geq 1, s,t> 0 \) there is a map \( g \) with 
\(
\ell_1=s, \ell_2=t,  k_1=\beta/t,  k_2=\beta/s,
\)
which implies that \( \beta^-= \ell_1 k_2=  \ell_2 k_1 = \beta^+\), and thus  \( g\in \mathfrak{F}^* \).

\section{The Induced Map}
\label{sec:recall}

In this section we recall some details of the construction of the induced map carried out in \cite{CoaLuzMub22} and a key estimate from \cite{CoaLuzMub22}, see  Proposition \ref{prop:tail-of-tau} below, which will play a crucial role in our proofs. 

\subsection{Topological construction}
We recall first of all from \cite{CoaLuzMub22} the topological structure of the \emph{first return maps} on  the intervals  \( \Delta_0^-, \Delta_0^+ \) defined in  \eqref{eq:Delta0}.  From  the definitions of the sets \( \Delta_{n}^{\pm}\) and \( \delta_{n}^{\pm}\) in \eqref{eq:Delta} and \eqref{eq:delta},  and from the fact that each branch of \( g \) is a \( C^{2}\) diffeomorphism, it follows  that for every  \( n \geq 1 \),  the maps
\(
g:\delta_{n}^{-} \to  \Delta_{n-1}^{+}
\)
and
\(
g:\delta_{n}^{+} \to  \Delta_{n-1}^{-}
\)
are \( C^{2}\) diffeomorphisms, and, for \( n \geq 2 \), the same is true for the maps
\(
g^{n-1}: \Delta_{n-1}^{-}\to \Delta_0^{-},
\) and
\(
g^{n-1}: \Delta_{n-1}^{+}\to \Delta_0^{+},
\)
and therefore  for every \( n \geq 1\), the maps
\(
	g^{n}: \delta_n^{-} \to \Delta_0^+
\) and 
\( 
	g^{n}: \delta_n^{+} \to \Delta_0^-
\)
are  \( C^{2} \) diffeomorphisms. We can therefore define two \emph{full branch}  maps
\( 
	\widetilde G^-:\Delta_{0}^{-} \to \Delta_{0}^{+} 
	\) and \( 
	\widetilde G^+:\Delta_{0}^{+} \to \Delta_{0}^{-}
\) by \(  \widetilde G^\pm|_{\delta_{n}^{\pm} } :=  g^{n}.
\) 
Then  for every \(i, j \geq 1\) we let
\begin{equation}
	\label{eq:def-delta-i-j}
	\delta_{i,j}^{-} := g^{-i}(\delta^+_j) \cap \delta^-_i
	\quand \delta^{+}_{i,j} := g^{-i}(\delta^-_j) \cap \delta^{+}_i
\end{equation}
Then, for  \( i \geq 1\), the sets
\(
\{ \delta_{i,j}^{-}\}_{j\geq 1}
\)
and
\(
\{ \delta_{i,j}^{+}\}_{j\geq 1}
\)
are partitions of \( \delta_i^-\) and \( \delta_i^+\)  respectively and so
\(
	\mathscr P^- :=  \{ \delta_{i,j}^{-}\}_{i,j\geq 1}
\) and \( 
	\mathscr P^+ :=  \{ \delta_{i,j}^{+}\}_{i,j\geq 1}
\)
are partitions of \( \Delta_0^-, \Delta_0^+\)  respectively, with the property that for every \( i,j \geq 1\), the maps \( 
	g^{i+j}: \delta_{i,j}^{-} \to \Delta_{0}^{-}
\) 
and \( 
	g^{i+j}: \delta_{i,j}^{+} \to \Delta_{0}^{+}
\) 
are \( C^2\) diffeomorphisms. Notice that  \( i+ j \) is the \emph{first return time} of points in \( \delta_{i,j}^{-} \) and \( \delta_{i,j}^{+} \)  to \( \Delta_{0}^{-} \) and  \( \Delta_{0}^{+} \) respectively, and we have thus constructed \emph{two full branch first return induced maps}
\( 
	G^-:=\widetilde G^+ \circ \widetilde G^- :\Delta_{0}^{-} \to \Delta_{0}^{-} 
	\) 
	and 
	\(  G^+:=\widetilde G^- \circ \widetilde G^+ :\Delta_{0}^{+} \to \Delta_{0}^{+}.
\) 
for which we have
\(
G^-|_{\delta_{i,j}^{-} }= g^{i+j}
\) and \( G^+|_{\delta_{i,j}^{+} }= g^{i+j}.
\)

We now focus on one of these two full branch first return maps, for definiteness let's say \( G^{-}\) (but we could just as well take \( G^{+}\)) and for simplicity omit the superscript from the notation and  write  
\begin{equation}\label{eq:noindex}
\Delta_{0} := \Delta_{0}^{-}, \quad \delta_{i,j}  := \delta_{i,j}^{-},  \quad G :=  G^{-}. 
\end{equation}
It is proved in \cite{CoaLuzMub22} that \( G: \Delta_{0}\to \Delta_{0} \) is a full branch \emph{Gibbs-Markov} map with respect to the partition \( \{ \delta_{i,j}  \} \), and  therefore admits an \emph{invariant ergodic probability measure} \( \hat \mu \) which is \emph{equivalent to Lebesgue} and with Lipschitz continuous  density. 
Then,  by standard results,  the \emph{induced measure}
\begin{equation}
  \label{eq:mu}
\mu  \coloneqq \sum_{n=0}^{\infty} 
  g^n_*(\hat\mu|\{\tau \geq  n\})
\end{equation}
is a \emph{sigma-finite, ergodic,  \( g\)-invariant} measure which, since   \( \bigcup g^{n}(\{\tau \geq  n \}) = I \ (\text{mod} \ 0)\) by construction, is \emph{equivalent to Lebesgue}.   It is easy to check that \( \mu(I) < \infty\) if and only if \( \tau \in L^{1}(\hat\mu)\), and it follows from Proposition 2.6 of \cite{CoaLuzMub22} (which we recall in Proposition \ref{prop:tail-of-tau} below) that \( \tau \in L^{1}(\hat\mu)\) if and only if  \( \beta\in [0,1)\), i.e. if and only if \( g\in \mathfrak F \), as already mentioned above. 
Moreover, since \( G \) is a \emph{first return} induced map, the measure \( \mu \) does not add any measure to the inducing domain and therefore 
\( 
\mu|\Delta_{0} = \hat \mu.
\)

\subsection{Inducing Times}
Our arguments revolve around the \emph{distribution} of  various \emph{observables} on \( \Delta_{0}\) with respect to the probability \( \hat \mu \). First we define   \( \tau^\pm(x), \tau: \Delta_0 \to \mathbb N \)  by
\begin{equation}\label{eq:def-tau-pm}
	\tau^+(x):= \#\{1\leq i \leq \tau: g^i(x)\in I_+ \}
	\quad
	\tau^-(x):= \#\{1\leq j \leq \tau: g^j(x)\in I_- \},
	\quad \tau \coloneqq \tau^+ + \tau^- 
\end{equation}
where \( \tau^\pm \)  \emph{count  the number of iterates of \( x \) in \( I_-,  I_+\) respectively before returning to~\( \Delta_0\)}. 
Notice that \( 
\tau^+|_{\delta_{i,j}} \equiv i \) and  \(  \tau^-|_{\delta_{i,j}} \equiv j
\),  that 
	\( \tau|_{\delta_{i,j}} \equiv i+j  \)  is exactly  the \emph{first return time} to \(\Delta_0\). 
The following key technical result from \cite{CoaLuzMub22} gives the distribution of these functions. We say  that \( f \sim g \)  if \( f(t) / g(t) \to 1 \) as \( t \to \infty \).

 \begin{prop}[{\cite{CoaLuzMub22}*{Proposition 2.6}}]
  \label{prop:tail-of-tau}
  There exist constants \( C_+, C_- > 0 \) such that
  \begin{equation}
		\label{eq:dist-of-tau-p}
		  \hat \mu( \tau^+ > t ) \sim C_+ t^{-1/\beta^+},    \qquad \hat \mu( \tau^- > t ) \sim C_- t^{-1/\beta^-}, 
		  \qquad   \hat \mu( \tau > t ) \sim (C_+ + C_- )t^{-1/\beta}.
		  \end{equation}
\end{prop}	
We will also be interested in the associated Birkhoff sums  \( \tau_{k}^{\pm}, \tau_{k}: \Delta_{0} \to \mathbb{N} \), under the \emph{induced map} \( G \),  defined by
\begin{equation}\label{eq:deftaupm}
 \tau_{k}^{-} \coloneqq \sum_{ \ell = 0 }^{ k - 1 } \tau^{-} \circ G^{ \ell},
 \qquad 
 \tau_{k}^{+} \coloneqq \sum_{ \ell = 0 }^{ k - 1 } \tau^{+} \circ G^{ \ell}, 
 \qquad 
 \tau_k \coloneqq \tau_k^+ + \tau_k^-.
\end{equation}
These give us the total time which points spend in the left and right intervals after \( k \) iterations of the induced map. 
For future reference, note that \(  \tau_{k}^{-},  \tau_{k}^{+}, \tau_k\) are Birkhoff sums of  \(  \tau^{-},  \tau^{+}, \tau \) respectively, along the orbit of a point under the induced map \( G: \Delta_0\to\Delta_0 \) and therefore, by ergodicity and invariance of the probability measure \( \hat\mu\) under \( G \), since they are all non-negative observables, 
\begin{equation}\label{eq:int}
 \frac{\tau^{-}_{k}}{k} \to \int \tau^{-}d\hat\mu, 
 \qquad\qquad 
  \frac{\tau^{+}_{k}}{k} \to \int \tau^{+}d\hat\mu, 
  \qquad\qquad 
   \frac{\tau_{k}}{k} \to \int \tau d\hat\mu,
\end{equation}
as \( k \to \infty\),  irrespective of whether the integrals are finite or not, for \( \hat\mu\) almost every \( x\in \Delta_0\). Therefore, since \( \hat\mu\) is equivalent to Lebesgue,  for  Lebesgue almost every \( x\in \Delta_0\). 
  Proposition \ref{prop:tail-of-tau} implies  that \( 
	 \tau^-, \tau^+, \tau \in L^1\) if and only if \(\beta^-,  \beta^+,  \beta \in [0,1)\)
respectively,  and therefore, from \eqref{eq:int} we have that  	  if \( \beta^-,  \beta^+,  \beta \in [0,1)\),  then we have respectively that 
\begin{equation}\label{eq:divbeta<1}
\frac{\tau^{-}_{k}}{k} \to \smallint \tau^{-}d\hat\mu < \infty, 
\qquad \qquad  \frac{\tau^{+}_{k}}{k}\to  \smallint \tau^{+}d\hat\mu < \infty,  
\qquad \qquad
\frac{\tau_{k}}{k}\to  \smallint \tau d\hat\mu < \infty,
\end{equation}
and if \( \beta^-,  \beta^+,  \beta \geq 1\), then we have respectively that 
\begin{equation}\label{eq:div}
\frac{\tau^{-}_{k}}{k} \to \infty, 
\qquad \qquad \frac{\tau^{+}_{k}}{k}\to \infty, 
\qquad  \qquad  \frac{\tau_{k}}{k}\to \infty. 
\end{equation}

\section{Proof of Theorem \ref{thm:phys-measures}}
\label{sec:physical}

We split the proof of Theorem \ref{thm:phys-measures} into 3 parts. First we show that Lebesgue almost every point spends asymptotically all its time in either \(I^{-}\) or \( I^{+}\). Then we show that in fact such orbits spend most of the time in arbitrarily small neighbourhoods of the corresponding fixed points \( -1\) and \( + 1\) when measured along the subsequence \( \tau_{k}\). Finally we show that this implies that the same holds for the full sequence of iterates, thus proving the Theorem. 

\subsection{Statistics of orbits in \(I^{-}\) and \( I^{+}\)}

We now go into a bit more detail on the behaviour of the induced observables \(  \tau^-_k, \tau^+_k, \tau_k\). Recall that by definition  \( \tau_k= \tau^+_k+ \tau^-_k\) and therefore  
\begin{equation}\label{eq:tauk1}
\frac{\tau^-_k}{\tau_k}  +  \frac{\tau^+_k}{\tau_k} = \frac{\tau_{k}^{-}+ \tau_{k}^+ }{ \tau_k} = \frac{\tau_k}{\tau_k} =1
\end{equation}
where \( {\tau^-_k}/{\tau_k} \) and \( {\tau^+_k}/{\tau_k} \) are simply the proportion of time that the orbit of \( x \) spends on the left and right intervals respectively in its first \( \tau_k\) iterates corresponding to \( k \) iterations of the induced map. 
The main result of this section shows that when \( \beta\geq 1 \), the largest of \( \beta^-\) and \( \beta^+\) ``gets everything''. 

\begin{prop} 
	\label{prop:existence-of-physical-meaures}
Suppose \( \beta \geq 1\).	 Then 
	\begin{equation}\label{eq:existence-of-physical-meaures}
		\beta^-> \beta^+ \implies \frac{\tau^-_k}{\tau_k} \to 1	
		\quad	\quand \quad 
			\beta^+ > \beta^- \implies \frac{\tau^+_k}{\tau_k} \to 1 
	\end{equation}
	for Lebesgue almost every point \( x\in \Delta_0\). 
\end{prop}
It then follows of course from \eqref{eq:tauk1} and \eqref{eq:existence-of-physical-meaures} that when  \( \beta \geq 1\) we have also 
\(
	\beta^-> \beta^+ \implies  {\tau^+_k}/{\tau_k} \to 0	
\) 
and \( 	\beta^+ > \beta^- \implies  {\tau^-_k}/{\tau_k} \to 0.\) so Proposition \ref{prop:existence-of-physical-meaures} implies that whenever at least one of \( \beta_1, \beta_2\)  is \( \geq 1 \) and \( \beta_1\neq \beta_2\), Lebesgue almost every point spends asymptotically all its time either in the left or right interval.

\begin{proof}[Proof of Proposition \ref{prop:existence-of-physical-meaures}]
We will prove the result when \( \beta^{-}> \beta^{+} \), the case \( \beta^{+}> \beta^{-} \) follows by exactly the same arguments. Notice first of all that from \eqref{eq:tauk1} we have
	\[
	1= \frac{\tau_{k}^{-}+ \tau_{k}^+ }{ \tau_k}
	=  \frac{ \tau_{k}^- }{ \tau_k } +  \frac{ \tau_k^+}{ \tau_k } 
	= \frac{ \tau_{k}^- }{ \tau_k } \left( 1 + \frac{ \tau_k^+ }{ \tau_k^- } \right). 
	\]
and it is therefore sufficient to show that 
\begin{equation}\label{eq:equiv}
\frac{\tau^{+}_k}{ \tau_k^-} \to 0 
\end{equation}
as this implies 
\( { \tau_{k}^- }/{ \tau_k }  \to 1 \) as required. 
By assumption \( \beta\geq 1 \) and therefore \( \beta^{-}\geq 1 \), and it is therefore sufficient to consider the following   3 subcases.

\begin{description}
\item[\textbf{1)}]  \( \beta^{-}\geq 1 > \beta^+ > 0\).
\end{description}

In this case we can write 
\[
\frac{\tau^{+}_k}{ \tau_k^-}  = \frac{\tau^{+}_k}{ k}  \frac{k}{ \tau_k^-} 
\]
By  \eqref{eq:divbeta<1} we have that \( {\tau^{+}_k}/{ k}  \) is bounded and by   \eqref{eq:div}  we have  that  \( {k}/{ \tau_k^-} \to 0 \),  implying that \( {\tau^{+}_k}/{ \tau_k^-} \to 0  \) and thus giving \eqref{eq:equiv}. 

\begin{description}
\item{\textbf{2)}  \( \beta^{-}> \beta^+ >    1\) .} 
\end{description}

In this case, the statements in \eqref{eq:divbeta<1} and \eqref{eq:div} do not allow us to immediately draw any  definite conclusions and we need to refer to a non-trivial result of \cite{GalHolPer21} which applies precisely to our case. Indeed,  since  the map \( G \) is Gibbs-Markov, the functions \( \tau^{\pm} \) are constant on each partition element \( \delta_{i,j} \), and  the distribution of \( \tau^{\pm} \) is given by \eqref{eq:dist-of-tau-p}, Proposition 2.8 of  \cite{GalHolPer21} applied to our setting satisfies  gives the following result. 

\begin{lem}[[Proposition 2.8]\cite{GalHolPer21} ]
	\label{lem:holland}
For all \( \epsilon > 0 \), 
\[
\beta^- > 1 \implies 
k^{\beta^{-}-\epsilon} \lesssim \tau^{-}_{k}\lesssim k^{\beta^{-}+\epsilon}
\quand 
\beta^+ > 1 \implies 
k^{\beta^{+}-\epsilon} \lesssim \tau^{+}_{k}\lesssim k^{\beta^{+}+\epsilon}.
\]
\end{lem}
This  implies 
\(
	{ \tau^{+}_k }/{ \tau_k^-} \lesssim  {k^{\beta^{+}+\epsilon} }/{k^{\beta^{-}-\epsilon}} 
	   \)
almost surely, and thus implies \eqref{eq:equiv} if \( \epsilon \) is sufficiently small.

\begin{description}
\item{\textbf{3)} \( \beta^{-}> \beta^+ =   1\).} 
\end{description}

Lemma \ref{lem:holland} requires \( \beta^{-}\) and \( \beta^{+}\) to be \emph{strictly} greater than 1 and therefore we cannot apply the argument above completely to this setting but we can conclude that 
	 \begin{equation}\label{eq:taukbound}
	 \frac{ \tau^{+}_k }{ \tau_k^-} \lesssim  \frac{\tau_k^{+}}{k^{\beta^{-}-\epsilon}}
	   \end{equation}
Letting 
\(
q:= {1}/{(\beta^{-}-\epsilon)}
\) and using the definition of \( \tau_k^{+}\) in \eqref{eq:deftaupm}  it will be convenient to write this as 
 \begin{equation}\label{eq:taukbound2}
	 \frac{ \tau^{+}_k }{ \tau_k^-} \lesssim  \frac{\tau_k^{+}}{k^{1/q}}
	 = \frac{\tau^{+}+(\tau^{+} \circ G) + \cdots + (\tau^{+} \circ G^{k-1} )}{k^{1/q}}
	   \end{equation}
Since \( \hat \mu \) is \( G \) invariant, the summands \( \tau^+ \circ G^\ell \) are identically distributed	  and 
Proposition \ref{prop:tail-of-tau}  gives 
\[
 \mu ( (\tau^{+})^{q}  > t ) = \mu (\tau^{+}> t^{1/q} )  \sim C_{+} t^{-1/q\beta^+} =
  C_{+} t^{-(\beta^{-}-\epsilon)/\beta^+}
\] 
which implies that   \( (\beta^{-}-\epsilon)/\beta^+> 1 \) and therefore  \( \tau^+ \in L^q ( \hat \mu ) \) if  \( \epsilon \) is small enough. Then we can apply  the following classical and remarkable result. 
\begin{lem}\cite{Saw66}*{Corollary to Lemma 3}
\label{lem:saw}
Let \( (X, \mu) \) be a probability space and suppose  \( \varphi_{n}  \) are identically distributed random variables with  \( \varphi_{n} \in L^{q} \) for some  \( q \in (0,1)\). 
Then. \( \mu \)-almost surely, 
\[
\frac{\varphi_{1}+ \cdots + \varphi_{n}}{n^{1/q}} \to 0. 
\]
\end{lem}
Applying  \ref{lem:saw} to  \eqref{eq:taukbound2} implies  \eqref{eq:equiv} and thus completes the proof.
\end{proof}

\subsection{Statistics of orbits near the fixed points for the subsequence \( \tau_{k}\) }
Proposition \ref{prop:existence-of-physical-meaures} tells us that depending on the  relative values of \( \beta^{-}, \beta^{+}\) orbits spend asymptotically all their time inside either the left or right subintervals \( I^{-}, I^{+}\).  We are however especially interested in how much time orbits spend close to the two fixed points and in this section we will show that actually most of the time spent in these intervals is spent in arbitrarily small neighbourhoods of the corresponding fixed points. 
To formalize  this statement,   for \( \eps > 0 \) we define the intervals
\begin{equation}
	U^{+}_{\eps} \coloneqq ( 1 - \eps, 1 ), \quand
	U^{-}_{\eps} \coloneqq ( -1 , -1 + \eps ),
\end{equation}
and then define the functions \( S_{n,\eps}^{\pm} : [-1,1] \to \mathbb{N} \) and  \( 	  S_{n,\eps} : [-1,1] \to \mathbb{N} \)  by 
\begin{equation}
	\label{eq:def-S-n-eps-pm}
        S_{n, \eps}^{-} \coloneqq \sum_{k = 0}^{ n - 1 } \mathbb{1}_{U_{\eps}^{-}} \circ g^{k},  
	\qquad
	        S_{n, \eps}^{+} \coloneqq \sum_{k = 0}^{ n - 1 } {\mathbb{1}}_{U_{\eps}^{+}} \circ g^{k}, 
	        \quand 
S_{n, \eps}=  S_{n,\eps}^{-} + S_{n,\eps}^{+}.          
\end{equation}
The functions \( S_{n,\eps}^{-}\) and \( S_{n,\eps}^{+}\) simply count the number of iterates of a point which belong to the neighbourhoods  \( U^{-}_{\eps} \) or \( U^{+}_{\eps}\) respectively, in the first \( n \) iterates. 

\begin{prop}\label{lem:S-vs-tau-k}
For every  \( \varepsilon > 0 \) and Lebesgue  almost-every \( x \in \Delta_0 \),
	\begin{equation}
		\label{eq:S-tau-k-pm-vs-tau-k-pm}
		\beta^{-} \geq  1 \implies  \frac{S_{\tau_k, \eps}^{-}}{\tau_k^-}\to 1,
		\qquad 
		\beta^{+}\geq  1 \implies  \frac{S_{\tau_k, \eps}^{+}}{\tau_k^+}\to 1,
		\qquad 
		\beta \geq  1 
	\implies 
		\frac{S_{\tau_k, \eps}}{\tau_k} \to 1. 
	\end{equation}
\end{prop}

\begin{proof}
Recall the definition of the partitions \( \{ \Delta_{n}^{\pm} \} \) in \eqref{eq:Delta} and,  
for  \( \eps > 0 \), let 
 \[
  N_{\varepsilon}^{\pm} \coloneqq \max \left\{ N : U_{\eps}^{\pm} \subseteq \bigcup_{k = N}^{\infty} \Delta_{k}^\pm \right\}.
  \]
Then it is easy to see from the definition of the partition \( \{ \delta_{i,j} \} \) in  \eqref{eq:def-delta-i-j} and \eqref{eq:noindex} and the properties of the induced map  that all points in \( \delta_{i,j}\) with \( i, j \geq N_{\varepsilon}\) will spend all but \( N_{\varepsilon}^{+} \) and \( N_{\varepsilon}^{-} \) iterates inside  \(U_{\eps}^{-}\) and  \(U_{\eps}^{+}\) respectively before they return at time \( \tau = i+j\).  More formally,  if \( x \in \delta_{i,j} \) and \( i > N^{+} \) then \( g^{k} (x) \in U^+_{\eps} \) for \( k < i - N^{+} \) and \( g^{k} \not\in U^{+}_{\varepsilon} \) for \(  i - N^{+} < k \leq i \), similarly  if \( x \in \delta_{i,j} \) and \( j > N^{-} \) then \( g^{i + k} (x) \in U^+_{\eps} \) for \( 1 \leq k < j - N^{+} \) and \( g^{i + k} \not\in U^{+}_{\varepsilon} \) for \(  j - N^{+} < k \leq j \). Therefore, from the definitions in 
 \eqref{eq:def-tau-pm}, \eqref{eq:def-S-n-eps-pm} it follows that for every  \( x \in \Delta_0 \) we have
	\begin{equation}
		\label{eq:ineq-tau-S_n}
		\tau^+ (x)  - N_{\eps}^{+} - 1
		\leq S^{+}_{\tau (x),\varepsilon} (x) \leq \tau^{+} (x),\quand
		\tau^- (x)  - N_{\eps}^{-} - 1
		\leq S^{-}_{\tau (x),\varepsilon} (x) \leq \tau^{-} (x).
	\end{equation}
Notice that this holds even if \( \tau^{\pm}(x) \leq N_{\varepsilon}^{\pm}\) since then the left hand side of the corresponding inequality is negative. 	
	From the definition of \( \tau_{k} \) we can  write 
	\[
		S_{\tau_k(x),\varepsilon}^\pm (x) = S_{\tau(x), \eps}^\pm (x) + S_{\tau ( G(x) ), \eps }^{\pm} ( G( x ) ) + \cdots + S_{\tau (G^{k-1} (x) ), \eps}^{\pm} ( G^{k-1} (x) ).
	\]
	Applying the inequalities in \eqref{eq:ineq-tau-S_n} to each term in the sum above we get 
	\begin{equation}
		\label{eq:S-tau-k-vs-tau-k}
	\tau^\pm_k (x)  -  k ( N_{\eps}^\pm + 1 ) = 	\sum_{ m = 0 }^{ k  - 1 } \tau^\pm \circ G^m (x)  - k ( N_{\eps}^\pm + 1 ) \leq S_{\tau_k (x), \eps }^{ \pm } (x) \leq  \sum_{ m = 0 }^{ k  - 1 } \tau^\pm \circ G^m (x) = \tau^\pm_k (x) 
	\end{equation}
	and so, dividing through by \( \tau^\pm_k (x) \), gives 
	\[
		1 - \frac{ k }{ \tau^\pm_k (x) }( N^\pm_\eps + 1 ) \leq \frac{ S_{ \tau_k (x), \eps }^\pm (x) }{ \tau_k^{\pm}  (x) }  \leq 1.
	\]
From \eqref{eq:div} we have   \( k/{\tau_k^\pm (x) } \to 0\) 	for \( \hat \mu \) almost every \( x \in \Delta_0 \), yielding \eqref{eq:S-tau-k-pm-vs-tau-k-pm}. From \eqref{eq:S-tau-k-vs-tau-k} we also get 
	\[
		\tau_k(x) - k ( N_\eps^+ + N_{\eps}^- + 2 ) \leq S_{\tau_k(x), \eps}^{+} (x) + S_{\tau_k(x), \eps}^-(x) \leq \tau_k(x).
	\]
	Dividing through by \( \tau_k(x) \) and applying \eqref{eq:div}  as above, we get \eqref{eq:S-tau-k-pm-vs-tau-k-pm} and complete the~proof.
\end{proof}
\subsection{Statistics of orbits near the fixed points}

We are now ready to prove Theorem \ref{thm:phys-measures}. 

\begin{proof}[Proof of Theorem \ref{thm:phys-measures}]
We will show that for every  \( \varepsilon > 0 \) and  Lebesgue almost every point \( x\in \Delta_0\),
	\begin{equation}\label{eq:seq}
	\beta^{-} >  \beta^+ \implies  \frac{S_{n, \eps}^{-}}{n}\to 1,
		\qquad 
		\beta^{+} >  \beta^- \implies  \frac{S_{n, \eps}^{+}}{n}\to 1.
	\end{equation}
This clearly implies the statement of 	Theorem \ref{thm:phys-measures}. 
Notice first of all that 	from Propositions \ref{prop:existence-of-physical-meaures} and~\ref{lem:S-vs-tau-k} we immediately get that for every  \( \varepsilon > 0 \) and  Lebesgue almost every point \( x\in \Delta_0\),
	\begin{equation}\label{eq:existence-of-physical-measures}
		\beta^-> \beta^+ \implies \frac{S_{\tau_k, \eps}^{-}}{\tau_k} \to 1	
		\quad	\quand \quad 
			\beta^+ > \beta^- \implies \frac{S_{\tau_k, \eps}^{+}}{\tau_k} \to 1. 
	\end{equation}
We therefore just   need to replace the convergence along the subsequence \( \tau_{k}\) with convergence along the full sequence. 
Suppose first that \( \beta^{+} >  \beta^- \). Let \( x \in \Delta_{0}^{-}\) and consider the sequence of iterates \( g^{i}(x)\) for \( 1\leq i \leq \tau(x)\). Recall from the construction of the induced map that the iterates for which \( g^{i}(x) \in U^{-}_{\varepsilon} \) all lie at the ``beginning'' of the sequence, i.e. once the orbit leaves the neighbourhood \( U^{-}_{\varepsilon}  \) it cannot return to it before the next return to \( \Delta^{-}_{0}\). More formally, either  \( g^{i}(x)\notin U^{-}_{\varepsilon}\) for all \( 0\leq i \leq  \tau(x)\) (i.e. the finite piece of orbit never enters \( U^{-}_{\varepsilon} \),  or there exists an integer \( 1\leq m_{\varepsilon} < \tau(x)\) such that  \( g^{i}(x)\in U^{-}_{\varepsilon}\) for all \( 0 <  i \leq m_{\varepsilon}\) and \( g^{i}(x)\notin U^{-}_{\varepsilon}\) for all \( m_{\varepsilon}\leq i \leq \tau(x)\). This means that the ratio \( {S_{n, \eps}^{+}}/{n} \) is always larger \emph{in-between} returns that at the \emph{following} return or, more formally,  letting  \( k \geq 1\) be the smallest integer such that \( \tau_{k}\geq n\), we have   
\[
1\geq \frac{S_{n, \eps}^{+}}{n} \geq \frac{S_{\tau_{k}, \eps}^{+}}{\tau_{k}}\
\]
By  \eqref{eq:existence-of-physical-measures} this  implies \eqref{eq:seq}  in the case \( \beta^{+} >  \beta^- \). The case \( \beta^{-}> \beta^{+}\) follows replacing \( - \) by \( + \) in \eqref{eq:noindex} and carrying out exactly the same argument. 
\end{proof}

\section{Proof of Theorem \ref{thm:no-phys-measures}}
In this section we assume throughout that  \( g\in \mathfrak{F}_* \) and therefore \( \beta^{-}=\beta^{+}=\beta\). The main technical result in the proof of Theorem \ref{thm:no-phys-measures} is the following. 

\begin{prop}\label{prop:limsupliminf}
  For  Lebesgue almost every \( x\in \Delta_{0} \)  we have 
  \begin{equation}
    \label{eq:tilde-tau-limsup-liminf>1}
    \limsup_{k\to\infty} \frac{ S_{n,\eps}^+ ( x ) }{n} = 
    \limsup_{k \to\infty} \frac{ S_{n, \eps }^{-} ( x ) }{n} = 1.
  \end{equation}
\end{prop}

This immediately implies that the sequence \( \mu_{n}(x) \) does not converge and that therefore \( g \) does not admit any physical measures, and thus implies the first part of 
Theorem \ref{thm:no-phys-measures}. In Section \ref{subsec1} we reduce   the proof of Proposition~\ref{prop:limsupliminf} to a more technical statement formulated in Proposition \ref{prop:tau-p-over-tau-k1}. In Section \ref{sec:proof} we prove Proposition \ref{prop:tau-p-over-tau-k1} and finally, in Section \ref{sec:accpoints} we will  use \eqref{eq:tilde-tau-limsup-liminf>1} to show that the set of limit points of the sequence \( \mu_{n}(x) \)  consists of the set \( \Omega\), completing the proof of Theorem \ref{thm:no-phys-measures}.

\subsection{Using the Gibbs-Markov structure}\label{subsec1}
in this section we reduce the proof of Proposition~\ref{prop:limsupliminf} to a statement about the positive measure of certain sets which are defined in terms of the Gibbs-Markov structure of the map \( G \) and of the observable \( \tau^-, \tau^+, \tau\). 
First of all let  
\[
\mathcal{A}^{+} \coloneqq \left\{ x : \limsup_{k\to\infty} \frac{ \tau^{+}\circ G^k (x) }{ \tau_k (x) } = +\infty \right\}
\quand 
\mathcal{A}^{-} \coloneqq  \left\{ x : \limsup_{k\to\infty} \frac{ \tau^{-}\circ G^k (x) }{ \tau_k (x) } = +\infty \right\}.
\]
Notice that the ratio \( { \tau^{+}\circ G^k (x) }/{ \tau_k (x) } \) compares the value of \( \tau^{+}\) at the \(k\)'th return time, i.e. essentially the number of iterates  the orbit spends near the endpoint \( 1 \) between the \(k\)'th and \( k+1\)  return time,  with the total accumulated length \( \tau_{k}\) of the orbit up to it's \( k \)'th return. If this ratio is big it means that time spent near 1   dominates whatever behaviour the orbit may have exhibited up to that time, and for  points in \( \mathcal A^{+}\),  this ratio is \emph{arbitrarily large} infinitely often. Similar observations hold  for \( \mathcal A^-\) and therefore  points   \( x \in \mathcal A^{+}\cap \mathcal A^{-}\) satisfy \eqref{eq:tilde-tau-limsup-liminf>1}, as formalized in the following Lemma.

\begin{lem}\label{lem:a+limsup}
\[
x\in \mathcal A^{+} \Longrightarrow \limsup_{ n \to \infty} \frac{S_{n,\eps}^+}{ n} = 1
\quand 
x\in \mathcal A^{-} \Longrightarrow \limsup_{ n \to \infty} \frac{S_{n,\eps}^-}{ n} = 1
\]
In particular, if \( x \in \mathcal A^{+}\cap \mathcal A^{-}\) then  \( x \) satisfies \eqref{eq:tilde-tau-limsup-liminf>1}
\end{lem}

\begin{proof}
 Fixing 
\(\eps > 0\) and choosing the subsequence  \(n_k \coloneqq \tau_k + \tau^+\circ G^k \) we have 
\[
  \frac{ S_{n_k, \eps}^+ }{ n_k } =
  \frac{ \tau^+_k + \tau^+ \circ G^k  }{ \tau_k + \tau^+ \circ G^k }
  = \frac{ \tau^+_k }{ \tau_k^+ } \frac{ 1 } { 1 + \tau^+ \circ G^k / \tau_k } + \frac{ 1 }{ 1 + \tau_k / \tau^+\circ G^k }
\]
and so \( x\in \mathcal A^{+}\) implies  \(\limsup_{ n \to \infty} S_{n,\eps}^+ / n = 1\). 
Similarly, we get the statement for points in \( \mathcal A^-\).
\end{proof}

It follows from Lemma \ref{lem:a+limsup} that Proposition \ref{prop:limsupliminf} follows if we prove that 
\begin{equation} \label{prop:tau-p-over-tau-k}
 Leb (\mathcal{A}^{+}) = Leb (\mathcal{A}^{-})=1. 
 \end{equation}
Notice moreover that since the invariant measure \( \hat\mu\) is equivalent to Lebesgue on \( \Delta_{0}\) it is enough to show that \( \hat\mu(\mathcal A^{+}) = \hat\mu (\mathcal{A}^{-})=1 \). Moreover, both sets \( \mathcal A^{+}\) and \( \mathcal A^{-}\) are \emph{invariant} for the map \( G \), i.e. \( G^{-1}(\mathcal A^{\pm}) = \mathcal A^{\pm} \), and therefore, by ergodicity of \( \hat\mu \) it is sufficient to show that 
\begin{equation} \label{prop:tau-p-over-tau-kpositive}
\hat\mu (\mathcal{A}^{+}) > 0 \quand \hat\mu (\mathcal{A}^{-})>0
 \end{equation}
 since this implies  \eqref{prop:tau-p-over-tau-k}.  To prove \eqref{prop:tau-p-over-tau-kpositive} it is sufficient to show that there exist  subsets \( \widetilde{\mathcal A}^\pm\subseteq \mathcal{A}^{\pm}\)  with \( \hat\mu  (\widetilde{\mathcal{A}}^{\pm})>0 \). For simplicity we will just define the prove the result for \({\mathcal{A}}^{+}\) and \( \widetilde{\mathcal{A}}^{+}\), which therefore for convenience will will just denote by \(\widetilde{\mathcal{A}}\), as the  result for \({\mathcal{A}}^{-}\) and \( \widetilde{\mathcal{A}}^{-}\) follows by identical arguments. To define \( \widetilde{\mathcal A} \) we first need to make use of the fact that return time \( \tau \) satisfies a  stable law.

\begin{lem}
  \label{lem:stable-law}
  There is  a non-degenerate stable random variable \( Y \) such that for almost every \( t \in \mathbb{R} \)
  \[
    \lim_{ t\to \infty } \mu \left(
      \frac{ \tau_k } { k^\beta } - d_k \leq t 
    \right)
    = \mu ( Y \leq t ),
  \]
  where \( d_k = 0 \) if \( \beta > 1 \) and \( d_k \approx \log k \) if \( \beta = 1 \).  In particular, there is  \( T > 0 \)  such that \( \mu(  Y < T ) > 0 \).
\end{lem}

We  note that to  simplify the notation we use the symbol \( \approx\) to denote the fact that the ratio between quantities is uniformly bounded above and below by a constant independent   of \( k \).

\begin{proof}
  The fact that \( \tau \) satisfies stable law can be see from \cite[Theorem 6.1]{AarDen2001}, however the formulation given there it is not immediate that the sequence \( d_k  \) should grow like \( \log k \) when \( \beta = 1 \). So, instead we will argue by \cite[Theorem 1.5]{Gou2010}.
  Let \( X_0, X_1, \ldots  \) be a sequence of positive, independent and identically distributed random variables with the same distribution as \( \tau \): \( \mathbb{P} ( X_0 \leq t ) = \mu ( \tau \leq t ) \).
  From Proposition \ref{prop:tail-of-tau} we know that \( t^{1/\beta} \mathbb{P} ( X_0 \leq t ) \to C > 0 \) and so the classical probability literature (see for example \cite{Nol20})  informs us that there exists a sequence \( d_k \) such that  for almost every \( t \in \mathbb{R} \) we have 
    \(
    \mathbb{P} ( { \sum_{ j = 0 }^{ k - 1 } X_j }/{ k^{\beta} } - d_k \leq t )
    = \mathbb{P} ( Y \leq t ),
  \)
  where \( Y \) is a non-degenerate stable random variable. Moreover, we know from \cite{Nol20} that \( d_k= 0  \) if \( \beta > 1 \) and \( d_k \approx \log k \) if \( \beta = 1 \).
  As \( \tau \) is constant on the partition elements \( \delta_{i,j} \), and as \( F \) is a topologically mixing Gibbs-Markov map, we know from \cite[Theorem 1.5]{Gou2010} that \( \tau_k \) will satisfy the same distributional convergence as \( \sum_{ j = 0}^{ k - 1} X_j \).
Finally, as \( Y \) is non-degenerate we know that there must exists some \( T > 0  \) such that \( \mu ( Y < T ) > 0 \), 
  which yields the Lemma.
\end{proof}

Now let \( d_k \)  and \( T > 0 \) be as in Lemma \ref{lem:stable-law} and let \(a_k\) be a sequence such that 
\begin{equation}\label{eq:annonsum}
  a_k \to +\infty\
  \quand 
  \sum_{k = 1}^\infty \frac{1}{ a_{k}^{1/\beta} k ( T +  d_k ) } = + \infty
\end{equation}
(as \( d_k \lesssim \log (k ) \), regardless the value of \( \beta \), we could take for example \( a_{k}^{1/\beta} = \log\log k\)), and let 
\[
  A_{k}\coloneqq  
  \{ \tau_k < k^{\beta} (T + d_k ) \text{ and }   \tau^{+}\circ G^{k} > a_k k^{\beta} ( T+ d_k)
  \}
\quad \text{ and let } \quad 
  \widetilde{\mathcal A}\coloneqq
  \bigcap_{ n  = 1 }^{ \infty } \bigcup_{ k \geq n } A_k
\]
be the set of points which belong to infinitely many \( A_{k}\). We will prove the following 
\begin{prop}\label{prop:tau-p-over-tau-k1}
  \( \widetilde{\mathcal A} \subseteq \mathcal A^{+}\) and 
  \( \hat\mu(\widetilde{\mathcal A})>0\). 
\end{prop}

Proposition \ref{prop:tau-p-over-tau-k1} clearly implies 
\( \hat\mu(\mathcal A^{+}) >0\), and an identical argument gives  \( \hat\mu(\mathcal A^{-}) >0\), and therefore, via 
\eqref{prop:tau-p-over-tau-kpositive} and \eqref{prop:tau-p-over-tau-k}, implies Proposition \ref{prop:limsupliminf}. In the next section we prove Proposition~\ref{prop:tau-p-over-tau-k1}.

\subsection{Proof of Proposition \ref{prop:tau-p-over-tau-k1}}
\label{sec:proof}
We will first show the first part of the statement, that  \( \widetilde{\mathcal A} \subseteq \mathcal A^{+}\), which is essentially a direct consequence of the definitions.

\begin{lem}
 \( \widetilde{\mathcal A} \subseteq \mathcal A^{+}\) 
\end{lem}
\begin{proof}
If \(x \in  A_k\) then, by definition, as \( k \to \infty\) we have 
\[
  \frac{ \tau_k } { \tau^+ \circ G^k } < \frac{  k^{\beta} (T + d_k ) }{ a_k k^{\beta} ( T  + d_k)}
  =   \frac{1}{ a_k} \to 0.
\]
So,  when \( k \) is sufficiently large, the ratio \(   \tau_k/ \tau^+ \circ G^k  \) is arbitrarily small and  \( \tau^+ \circ G^k/ \tau_k \) is arbitrarily large. 
It follows that 
for \(x \in \widetilde{\mathcal A}\) we have   \(\limsup_{ k\to \infty } \tau^+ \circ G^k ( x ) / \tau_k ( x )= \infty\) and so  \( \widetilde{\mathcal A}\subseteq \mathcal A\).  
\end{proof}

To show that \( \hat\mu(\widetilde{\mathcal A})>0\), and thus complete the proof of Proposition  \ref{prop:tau-p-over-tau-k1}, 
we will use a version of the second Borel-Cantelli lemma, originally due to R\'enyi, which says~that
  \begin{equation}\label{eq:borel}
    \hat \mu ( A_k \cap A_n ) \lesssim  \hat\mu (A_k) \hat\mu(A_n) \quand  
    \sum_{ k = 1 }^{ \infty } \hat \mu ( A_k ) = + \infty
    \quad 
\Longrightarrow 
\quad 
\hat\mu(\widetilde{\mathcal A}) = 
\hat\mu \left( \bigcap_{ n = 1 }^{ \infty } \bigcup_{ k \geq n} A_k \right) > 0.
  \end{equation}
A more precise statement,  see \cite[Theorem 2.1]{Pet2002},  includes a an explicit lower bound on the measure of \( \widetilde{\mathcal A}\) but we will not need that here. To verify the assumptions of \eqref{eq:borel}, we  observe first of all that the sets \( A_{k}\) can be written as the intersection  
\[
  A_k = B_k \cap G^{-k} C_k 
\]
where 
\[
   B_k \coloneqq      
  \left\{  \tau_k <  k^{\beta} ( T   + d_k ) \right\}
\quand 
    C_k \coloneqq \left\{ \tau^+ > a_k k^\beta  ( T  + d_k ) \right\}.
\]
Notice that   \( B_{k}\) is a union of elements of the partition \( \mathcal P^{(k)}\) (on which \( \tau_{k}\) is constant) whereas \( \mathcal C_{k}\) is a union of elements of the partition \( \mathcal P \) (on which \( \tau^{+}\) is constant), and therefore \( G^{-k}\mathcal C_{k} \) is a union of elements of the partition \( \mathcal P^{(k+1)}\). 
In particular, \( A_{k}\) is a union of elements of \( \mathcal P^{(k+1)}\).

Our proof will use in a crucial way the standard bounded distortion property of Gibbs-Markov maps which implies that for any measurable set \( \tilde\omega\subset \Delta_{0}\), any \( k\geq 1\), and any \( \omega^{(k)}\in \mathcal P^{(k)}\) we have 
\begin{equation}\label{eq:dist}
\hat\mu (\omega^{(k)} \cap \tilde\omega) \approx  \hat\mu( \omega^{(k)}) \hat\mu(G^{k}(\tilde \omega)). \end{equation}
We recall that  we use the symbol \( \approx\) to denote that the ratio between two quantities is uniformly bounded above and below by a constant independent   of \( k \),  \( \omega^{(k)}\) or \( \tilde\omega\). 
The standard formulation of bounded distortion is 
\(
{\hat\mu ( \omega^{(k)} \cap \tilde\omega)}/{\hat\mu( \omega^{(k)}) } 
\approx 
{\hat\mu(G^{k}(\tilde\omega))}/{\hat\mu(G^{k}(\omega^{(k)}))}, 
\) see Corollary 3.11 in \cite{CoaLuzMub22},  but since \( \omega^{(k)}\in \mathcal P^{(k)}\), we have \( G^{k}(\omega^{(k)}) = \Delta_{0}\),  and since \( \hat\mu(\Delta_{0}) =1 \) this implies \eqref{eq:dist}. 
We will also need the following result which is where we use the convergence to a stable law, as  stated  in Lemma~\ref{lem:stable-law}. 

\begin{lem}\label{lem:uniformbound}
The measure \( \hat \mu (B_k )  \) is uniformly bounded below. 
\end{lem}

\begin{proof}
Notice first of all that  can re-write the condition defining the sets \( B_{k}\) as 
\[
B_{k} =  \left\{ \frac{ \tau_k }{ k^{\beta} } - d_k <  T \right\}.
\]
By Lemma \ref{lem:stable-law}, 
we have  \(\hat \mu ( B_k ) \to \hat \mu (  Y <  T  )\) and so our choice of \( T \) implies the statement. 
\end{proof}

We can now verify the conditions of the Borel-Cantelli lemma in \eqref{eq:borel} in the following two lemmas, thus completing the proof of Proposition \ref{prop:tau-p-over-tau-k1}, and therefore of  \eqref{prop:tau-p-over-tau-k} and  \eqref{eq:tilde-tau-limsup-liminf>1}.

\begin{lem} For any \(k\geq 1 \), 
\begin{equation}
  \label{eq:AnCn}
  \hat\mu (A_k) \approx  \hat \mu ( C_k )
 \quad\text{ and therefore } \quad \sum_{k=1}^{\infty} \hat \mu (A_{k})  =  +\infty.
\end{equation}
\end{lem}

\begin{proof}
It follows  from the definition of the set \( A_{k}\),  \eqref{eq:dist}, and Lemma \ref{lem:uniformbound}, that 
\begin{equation}
  \label{eq:AnBnCn}
  \hat\mu (A_k) = 
  \sum_{\substack{\omega^{(k)}\in \mathcal P^{(k)} \\ \omega^{(k)}\subset B_{k}}} 
 \hat\mu(\omega^{(k)} \cap G^{-k} C_{k} )
  \approx 
   \sum_{\substack{\omega^{(k)}\in \mathcal P^{(k)} \\ \omega^{(k)}\subset B_{k}}} 
 \hat\mu(\omega^{(k)}) \hat\mu (C_{k} )
  \approx
  \hat \mu (  B_k ) \hat \mu ( C_k ) \approx  \hat \mu ( C_k ), 
\end{equation}
which is the first statement in \eqref{eq:AnCn}. 
Now, from  \eqref{eq:dist-of-tau-p}, keeping in mind that  \( \beta^{+}=\beta^{-}=\beta \geq 1\),  we have  \( \hat \mu(\tau^{+}> t) \approx 1/t^{1/\beta}\) and therefore, as \( d_k \) is only non-zero when \( \beta = 1 \),
\(
  \hat\mu (A_k) \approx  \hat \mu ( C_k ) \approx 1/ (a_{k}^{1/\beta} k (T + d_k )^{1/\beta})
  \approx 1/ ( a_{k}^{1/\beta} k ( T + d_k ))
\) 
which, by \eqref{eq:annonsum}, implies \eqref{eq:AnCn},  completing the  proof. 
\end{proof}

\begin{lem} For every \( n,k \geq 1 \) with \( n \neq k \) we have 
\(
 \hat \mu ( A_n \cap A_k ) \lesssim \hat\mu (A_k) \hat\mu(A_n) 
 \)
\end{lem} 
\begin{proof}
For every \( n \geq 1 \) we have 
\(  \hat\mu( A_n)  = \hat\mu(B_n \cap G^{-n} C_n) \leq \hat\mu( G^{-n} C_n)  \) and so, 
supposing without loss of generality that \( k > n \) and letting \( k= n+m\) and using also that  \( \hat \mu  \) is \( G \)-invariant, we get
\[
  \hat \mu (A_n \cap A_{n+m})
\leq \hat \mu ( G^{-n} C_n \cap G^{-k} C_{k} ) = 
 \hat \mu( G^{-n} ( C_n \cap G^{-m} C_{k} )) = \hat\mu(C_n \cap G^{-m} C_{k}).
\]
Since  \( C_{k}\) is a union of elements of  \( \mathcal P\) and \( G^{-m}C_{k}\) is a union of elements of   \( \mathcal P^{(m+1)}\), by  \eqref{eq:dist} we get 
\[
\hat\mu(C_n \cap G^{-m} C_{k} ) = 
\sum_{\substack{\omega^{(m)}\in \mathcal P^{(m)} \\ \omega^{(m)}\subset C_{n}}}
\hat\mu(\omega^{(m)} \cap G^{-m} C_{k} )
\approx 
\sum_{\substack{\omega^{(m)}\in \mathcal P^{(m)} \\ \omega^{(m)}\subset C_{n}}}
 \hat\mu (\omega^{(m)}) \hat\mu (C_{k})
 \leq  \hat\mu(C_{n}) \hat\mu (C_{k}).
\]
Combining these two expressions and using again the first part of \eqref{eq:AnCn} gives the statement. 
\end{proof}

\subsection{The space of accumulation points of \( \mu_{n}(x)\)}
\label{sec:accpoints}
In the previous subsection we have shown that \eqref{eq:tilde-tau-limsup-liminf>1} holds for Lebesgue almost every \( x \) and that therefore in particular the sequence \( \mu_{n}(x) \) does not converge. To complete the proof of Theorem \ref{thm:no-phys-measures} it therefore just remains to show that every measure \( \nu_{p}\), for every \( p\in [0,1]\), actually occurs as a limit point of the sequence and that these are the only measures which can occur as such limit points. This is a consequence of the following two lemmas. The first is a general result which applies whenever \( \beta\geq 1\) and which essentially  almost every orbit eventually spends most of it's time in arbitrarily small neighbourhoods of 1 and/or -1.

\begin{lem}
\label{lem:occ-tim-middle}
For almost every \(x\), and for every \(0 < \eps < 1\), 
\begin{equation}
  \label{eq:sum-to-1}
  \frac{ S_{n, \eps}^{+} (x) + S_{n, \eps}^{-} (x) } { n } \to 1.
\end{equation}
\end{lem}

\begin{proof}
Let \(K_{\eps} = [-1,1] \setminus \left ( U_{\eps}^+ \cup U_{\eps}^- \right)\) and let 
\begin{equation}\label{eq:sum-to-1a}
  K_{n, \eps}(x) \coloneqq \sum_{ k = 0 }^{ n - 1 } \mathbb{1}_{K_{\eps}} \circ  g^k ( x ) = n- (S_{n, \eps}^{+} (x) + S_{n, \eps}^{-} (x) )
 \end{equation}
  be the number of times the orbit of \( x \) belongs to \( K_{\eps}\) before time \( n \). 
Notice that there exists  \(N_\eps\),  depending only on \( \eps\),  such that 
\(
K_{\eps} \subset \bigcup_{ j = 0 }^{ N_{\eps} } \Delta_j^\pm
\)
 and so   the orbit of \(x\) enters \(K_{\eps}\) at most \(2 N_{\eps}\) times between one return and the next and so for every \( k\geq 1 \) we have \(  {K_{\tau_{k, \eps}}} \leq 2N_\eps k \). Thus,  letting  \(k_n\) be the unique integer such that  \(\tau_{k_n} (x) \leq n < \tau_{k_n + 1} (x)\) we have 
\begin{equation}\label{eq:sum-to-1b}
  \frac{K_{n, \eps}}{n}
  \leq 
 \frac{{K_{\tau_{k_{n}+1}, \eps}}}{\tau_{k_n}} \leq 2N_\eps \frac{k_n+1}{\tau_{k_{n}}}
 \end{equation}
Since \( \beta\geq 1\) we have that \( \tau_k/k \to \infty\), recall \eqref{eq:div}, and therefore  \( k/\tau_k \to 0 \) and so \( ({k_n}+1)/{\tau_{k_{n}}}\to 0\). 
 By \eqref{eq:sum-to-1b} this implies \( {K_{n, \eps}}/{n}\to 0\) and therefore, substituting into \eqref{eq:sum-to-1a} proves the Lemma. 
\end{proof}

The next lemma  is  an intuitive but non-trivial link between the limit points of \( \mu_n(x)\) and~\( S_{n,\eps}^+(x) / n\).

\begin{lem}
\label{lem:conv-to-nu-p}
 For all \( p\in [0,1]\), \(\mu_{n_k} ( x ) \to \nu_p\) if and only if \(S_{n_k,\eps}^+(x) / n_k  \to p\) for  all small \(\eps > 0\).
\end{lem}

\begin{proof}
First suppose that \(\mu_{ n _ k } (x) \to \nu_p\) and fix some small \( \eps>0\). 
Then, from the definition of the function \(  S_{n, \eps}^{+}  \) and the measure \( \mu_{n}(x)\) we have 

\[
\frac{S_{n_{k}, \eps}^{+} } {n_{k} }
= 
\frac{1}{n_{k}}  \sum_{k = 0}^{ n - 1 } \mathbb{1}_{U_{\eps}^{+}} \circ g^{k}
=
\frac{1}{n_{k}} \sum_{k = 0}^{ n - 1 } \int \mathbb{1}_{U_{\eps}^{+}} d\delta_{g^{i}(x)}
=
\int \mathbb{1}_{U_{\eps}^{+}} d\mu_{n_{k}}(x). 
  \]      
  Now let  \( \tilde \varphi_{\varepsilon}, \hat\varphi_{\varepsilon}\) be two continuous functions  such that \( \tilde \varphi_{\varepsilon} \leq \mathbb{1}_{U_{\eps}^{+}}  \leq  \hat\varphi_{\varepsilon}\) and which are equal to 1 in some neighbourhood of \( 1\), equal to 0 in some neighbourhood of \( -1 \). 
  Then       
\[
\int \tilde \varphi_{\varepsilon} d\mu_{n_{k}}(x) \leq 
\int \mathbb{1}_{U_{\eps}^{+}} d\mu_{n_{k}}(x)
\leq 
\int  \hat\varphi_{\varepsilon} d\mu_{n_{k}}(x)
\]
and, since \( \mu_{n_{k}}(x) \to \nu_{p}\) by assumption, 
\[
\int \tilde \varphi_{\varepsilon} d\mu_{n_{k}}(x)  \to \int \tilde \varphi_{\varepsilon} d\nu_{p} = p 
\quand 
\int \hat \varphi_{\varepsilon} d\mu_{n_{k}}(x)  \to \int \hat  \varphi_{\varepsilon} d\nu_{p} = p 
\]
which, substituting into the expressions above,  implies  \(S_{n_k,\eps}^+(x) / n_k  \to p\). 
Now suppose  that 
\begin{equation}\label{eq:convtop}
\frac{S_{n_k,\eps}^{+} (x)}{n_{k}} \to p 
\quand 
\frac{S_{n_k,\eps}^{-} (x)}{n_{k}}  \to 1-p
\end{equation}
for all \(\eps > 0\) (where the second statement follows by the first and the second part of \eqref{eq:tilde-tau-limsup-liminf>1}). To show that \( \mu_{n_{k}}\to \nu_{p}\) we will show that for any   continuous and bounded function  \(\varphi\), we have \( \int \varphi d\mu_{n_{k}} \to \varphi(-1)( 1-p ) + \varphi(1)p = \int \varphi d\nu_{p}\).  Note first that by definition of the measures \( \mu_{n_{k}}(x) \), 
\begin{equation}\label{eq:timeav}
\int \varphi d\mu_{n_{k}} =  \frac{1}{n_{k}} \sum_{ i = 0 }^{ n_{k} - 1 } \varphi \circ g^i (x).   
\end{equation}
Now, for any  \(\tilde \eps > 0\), let  \(\eps > 0\) be sufficiently small so that  
\begin{equation}\label{eq:closeto1}
\sup_{ y \in U^{\pm}_{\eps} } | \varphi(\pm 1)  - \varphi ( y ) | < \tilde\eps,
\end{equation}
and split up the time averages \eqref{eq:timeav} of \( \varphi\) as follows: 
\begin{equation}\label{eq:splitsum}
\frac{1}{n_{k}}
\sum_{ i = 0 }^{ n_{k} - 1 } \varphi \circ g^i (x) = 
 \frac{1}{n_{k}}\sum_{\substack{0\leq i < n_{k}: \\ g^{i}(x)\in U_{\eps}^{-}}} \varphi \circ g^i (x) 
 + 
 \frac{1}{n_{k}} \sum_{\substack{0\leq i < n_{k}: \\ g^{i}(x)\in U_{\eps}^{+}}} \varphi \circ g^i (x)  
  + 
 \frac{1}{n_{k}}  \sum_{\substack{0\leq i < n_{k}: \\ g^{i}(x)\notin U_{\eps}^{\pm}}}  \varphi \circ g^i (x). 
\end{equation}
For the first two terms of \eqref{eq:splitsum}, by \eqref{eq:closeto1}, we have
\[
(\varphi(\pm1) - \tilde\eps)  \frac{S_{n_k,\eps}^{\pm} (x)}{n_{k}}  
 \leq 
  \frac{1}{n_{k}} \sum_{\substack{0\leq i < n_{k}: \\ g^{i}(x)\in U_{\eps}^{\pm}}}  \varphi \circ g^i (x) 
  \leq (\varphi(\pm 1) + \tilde\eps) \frac{S_{n_k,\eps}^{\pm} (x)}{n_{k}}  
\]
which, by \eqref{eq:closeto1}, for \( n_{k} \) sufficiently large, gives 
\begin{equation}\label{eq:convnup1}
(\varphi(1) - 2 \tilde\eps) p
 \leq 
  \frac{1}{n_{k}} \sum_{\substack{0\leq i < n_{k}: \\ g^{i}(x)\in U_{\eps}^{+}}}  \varphi \circ g^i (x) 
  \leq (\varphi(1) + 2\tilde\eps) p
\end{equation}
and 
\begin{equation}\label{eq:convnup2}
(\varphi(-1) - 2 \tilde\eps) (1-p)
 \leq 
  \frac{1}{n_{k}} \sum_{\substack{0\leq i < n_{k}: \\ g^{i}(x)\in U_{\eps}^{-}}}  \varphi \circ g^i (x) 
  \leq (\varphi(-1) + 2\tilde\eps) (1-p)
\end{equation}
For the third term of \eqref{eq:splitsum} we have 
\[
-  \| \varphi \|_{ \infty} \frac{n_{k} - S_{n_k,\eps}^+(x) - S_{n_k,\eps}^{-} (x)}{n_{k}}\leq 
  \frac{1}{n_{k}}  \sum_{\substack{0\leq i < n_{k}: \\ g^{i}(x)\notin U_{\eps}^{\pm}}}  \varphi \circ g^i (x) \leq
   \| \varphi \|_{ \infty} \frac{n_{k} - S_{n_k,\eps}^+(x) - S_{n_k,\eps}^{-} (x)}{n_{k}}
\]
and therefore, for \( n_{k}\) sufficiently large, 
\begin{equation}\label{eq:convnup3}
-  \tilde\eps \| \varphi \|_{ \infty}  \leq 
  \frac{1}{n_{k}}  \sum_{\substack{0\leq i < n_{k}: \\ g^{i}(x)\notin U_{\eps}^{\pm}}}  \varphi \circ g^i (x) \leq
  \tilde\eps \| \varphi \|_{ \infty}. 
\end{equation}
Substituting \eqref{eq:convnup1}-\eqref{eq:convnup3} into \eqref{eq:splitsum}, and using  \eqref{eq:timeav}, we then get that for \( n_{k}\) sufficiently large we have
\[
\int \varphi d\mu_{n_{k}}
 \geq   (\varphi(1) - 2\tilde\eps) p + (\varphi(-1) - 2\tilde\eps) (1-p)- \tilde\eps \|\varphi \|_{ \infty}
\]
and 
\[
\int \varphi d\mu_{n_{k}}
 \leq   (\varphi(1) + 2\tilde\eps) p + (\varphi(-1) +  2\tilde\eps) (1-p) +  \tilde\eps \|\varphi \|_{ \infty}. 
\]
Since \( \tilde\eps\) can be chosen arbitrarily small by choosing \( n_{k}\) sufficiently large, it follows that 
\[
\int \varphi d\mu_{n_{k}} \to \varphi(1) p   +  \varphi(- 1) ( 1 -  p)
= \int \varphi d\nu_{p}
\]
and since \( \varphi\) was also arbitrary, this implies that \( \mu_{n_{k}} \to \nu_{p}\) as \( k \to \infty\). 
\end{proof}

\begin{proof}[Proof of Theorem \ref{thm:no-phys-measures}]

Equation  \eqref{eq:tilde-tau-limsup-liminf>1} immediately  implies that the sequence \( \mu_{n}(x)\) cannot converge and therefore \( g\) has no physical measures. Moreover, since  \eqref{eq:sum-to-1} holds for every  \( \eps>0\),  the limit points of \(\mu_n (x)\) must have support in \(\{+1, -1\}\) and therefore must be convex combinations of the form  \(  \nu_{p} \coloneqq p\delta_1 + ( 1 - p ) \delta_{-1} \) for \( p\in [0,1]\), and thus belong to the set  \( \Omega\),  recall \eqref{eq:Omega}. 
It therefore just remains to show that for any \( p \in [0,1]\), the measure \( \nu_{p}\) is a limit point of the sequence \( \mu_{n}(x)\). 

In view of Lemma \ref{lem:conv-to-nu-p}, it is sufficient to show that for every \(p\in [0,1]\) there exists a  subsequence 
\begin{equation}\label{eq:convp}
 n_{k}^{(p)} \to \infty \quad \text{ such that } \quad 
 \frac{S_{n_{k}^{(p)},\eps}^+(x)}{n_{k}^{(p)}}  \to p
 \end{equation}
  for  small \(\eps > 0\). By  \eqref{eq:tilde-tau-limsup-liminf>1}  this is true for \( p=0, 1\)   i.e. there exist subsequences \( n_{k}^{(0)}, n_{k}^{(1)}\) such that 
 \begin{equation}\label{eq:twosubs}
 \frac{S_{n_{k}^{(0)},\eps}^+(x)}{n_{k}^{(0)}}  \to 0
 \quand 
 \frac{S_{n_{k}^{(1)},\eps}^+(x)}{n_{k}^{(1)}}  \to 1.
 \end{equation}
Now notice that by definition of the sequence \( S^{+}_{n,\eps}(x)\), recall \eqref{eq:def-S-n-eps-pm},  we have that 
\[
S^{+}_{n+1,\eps}(x) - S^{+}_{n,\eps}(x) = 
\begin{cases} 
1& \text{ if } g^{n}(x) \in   U_{\eps}^{+} \\
0 &  \text{ if } g^{n}(x) \notin   U_{\eps}^{+} 
\end{cases}
\]
In particular, for any \( n\geq 1\) we have  \( S^{+}_{n+1,\eps}(x) - S^{+}_{n,\eps}(x) \leq 1 \) and therefore 
\begin{equation}\label{eq:onestep}
\left|
\frac{S^{+}_{n+1,\eps}(x)}{n+1} - \frac{S^{+}_{n,\eps}(x)}{n} 
\right| 
\leq \frac{1}{n}.
\end{equation}
This   means that for sufficiently large \( n \), the difference between two consecutive terms of the sequence \( {S^{+}_{n,\eps}(x)}/{n} \) is arbitrarily small. By \eqref{eq:twosubs} this sequence oscillates infinitely often between values arbitrarily close to 0 and arbitrarily close to 1 and therefore, by \eqref{eq:onestep}, for any \( p \in (0,1)\) and any \( \delta>0\) the sequence \( {S^{+}_{n,\eps}(x)}/{n} \) must enter the interval \( (p-\delta, p+ \delta) \) infinitely often. Since \( \delta \) is arbitrary,  there must be a subsequence converging to \( p \),  giving \eqref{eq:convp} and completing the proof. 
\end{proof}

\section{Proof of Theorem \ref{thm:density-main}}
\label{sec:density}

Throughout this section we fix \( f \in \widehat{ \mathfrak{F} } \) and let and let \( \ell_1,\ell_2,  k_1,k_2 , a_1,a_2,b_1,b_2, \iota  \)  be the corresponding parameters  as in~\eqref{eqn_1}. Then, given  \emph{arbitrary constants} \( \tilde\ell_{1}. \tilde\ell_{2}\geq 0\),  in Section \ref{sec:construction}   we will give a quite explicit construction of a map \( g \) which, in Section \ref{sec:fhat}, we will show  belongs to our class \( \mathfrak{F} \), with parameters \( \tilde{\ell}_1, \tilde{\ell}_2,   k_{1}, k_2, a_{1}, a_2, \tilde b_1, \tilde b_2 , \tilde\iota \), for appropriately chosen constants \(\tilde b_1, \tilde b_2, \tilde \iota \). In Section \ref{sec:crclose} we estimate the distance between \( f \) and  \( g \) in an appropriate topology and finally, in Section \ref{sec:conc}, we apply these estimates to the various cases required by Theorem  \ref{thm:density-main} and thus complete the proof.

\subsection{Construction of \( g \)}
\label{sec:construction}

 In Section \ref{sec:const-1} we describe the general construction of \( g \) and introduce the other parameters and functions on which \( g \) will depend. In Section \ref{sec:const-2} we then make some specific choices of the various parameters and functions involved in the construction. 

\subsubsection{General strategy for constructing the perturbation}
\label{sec:const-1}

Let \( f \in \widehat{\mathfrak{F}} \) and let the corresponding parameters  as in~\eqref{eqn_1} be \( \ell_1,\ell_2 \geq 0 \), \(k_1,k_2 , a_1,a_2,b_1,b_2  > 0 \). For any two constants \(\tilde\ell_{1} > 0 \) and  \(\tilde\ell_{2} > 0\), and any two compact intervals 
\[
[\tilde x_{1}, x_{1}]\subset U_{-1} \quand  [x_{2}, \tilde x_{2}]\subset U_{1},
\] 
we define functions
\( h_1 : U_{-1} \to [-1,1] \) and \( h_2 : U_{1} \to [-1,1] \) as follows. If \( \tilde\ell_{1}=\ell_{1}\) or \( \tilde\ell_{2}=\ell_{2}\) then we  let \( h_{1}=f|_{U_{-1}}\) and \( h_{2}=f|_{U_{1}}\) respectively. Otherwise, we let 
\begin{equation}
  \label{eq:def-h1}
  h_1(x) \coloneqq x + \tilde b_1 ( 1 + x )^{1 + \tilde{ \ell }_1}
  \quand
    h_2(x) \coloneqq x - \tilde b_2 ( 1 - x )^{1 + \tilde{ \ell }_2},
\end{equation}
where
\( \tilde b_1, \tilde b_2 > 0 \) are any constants such that
\begin{equation}
  \label{eq:cond-monotonicty}
  h_{1}(x) \leq  f(x) \quad \text{ on } \quad [x_1, \tilde x_1], 
  \quand
  h_{2}(x) \geq g(x) \text{ on } \quad [ \tilde x_2, x_2], 
\end{equation}
for every integer \( k\leq \lceil \ell_{1} \rceil \) and \( k\leq \lceil \ell_{2} \rceil \)  respectively. Note that if \( \tilde \ell_{1,2} \geq \ell_{1,2} \) then we can take \( \tilde b_{1,2} = b_{1,2} \) and the corresponding line of \eqref{eq:cond-monotonicty} will hold. Moreover, regardless of the relative values of \( \ell_1,\ell_2 \) and \( \tilde \ell_1, \tilde \ell_2 \), we \eqref{eq:cond-monotonicty} will always hold for all \( \tilde b_1, \tilde b_2 > 0 \) sufficiently small as we are only asking for the inequalities in \eqref{eq:cond-monotonicty} to be satisfied for \( x \) in compact intervals away from \( -1 \) and \( 1 \).
We let 
\[
\xi_{1}: [\tilde x_{1}, x_{1}] \to [0,1] \quand \xi_{2}: [ x_{2}, \tilde x_{2}]\to [0,1]
\] 
be  \( C^{\infty}\)  \emph{monotone increasing bijections} and
 define  \(g\) on the intervals \( [-1, 0) \) and \( [0, 1 ] \) by 
\begin{equation}
  \label{eq:def-g-1}
  g|_{[-1, 0]} (x) \coloneqq
  \begin{cases}
    h_1(x) & \text{if } x \in [-1, -\tilde x_1] \\
    h_1(x) + \xi_1 (x) ( f(x) - h_1(x) ) & \text{if } x \in (\tilde x_1, x_1) \\
    f(x) & \text{if } x \in [x_1, 0 ).
  \end{cases}
\end{equation}
and
\begin{equation}
  \label{eq:def-g-2}
  g|_{[0, 1 ]} (x) \coloneqq
  \begin{cases}
      f(x) & \text{if } x \in [0, x_{2} ]
      \\
    f(x) +  \xi_2 (x) (h_2 (x) - f(x) ) & \text{if } x \in (x_{2}, \tilde  x_{2}) \\
h_2(x) & \text{if } x \in [\tilde x_{2}, 1 ].
  \end{cases}
\end{equation}
Notice that \( g \) is equal to \( f \) outside \( U_{-1}\) and \( U_{1}\) and, apart from \( \tilde\ell_{1}, \tilde\ell_{2}\),  depends on 
the   two intervals \( [\tilde x_{1}, x_{1}], [ x_{2}, \tilde x_{2}]\), the constants \(  \tilde b_1,\tilde b_2\), and the functions   \( \xi_{1}, \xi_{2}\), which we explain below how to choose. 

\subsubsection{Choosing \( x_1,x_2,\tilde x_1, \tilde x_2, \tilde b_1,\tilde b_2,  \xi_{1}, \xi_{2}\)}
\label{sec:const-2}

We explain  how to make a specific choice of  the constants \( \tilde x_1, \tilde x_2, \tilde b_1,\tilde b_2\) and the functions \(  \xi_{1}, \xi_{2}\) depending on arbitrary \( x_1 \in U_{-1} \) and \( x_2 \in U_{1} \). 
 First of all we  define the \emph{affine orientation preserving bijections}  \( \eta_{1}: [\tilde x_{1}, x_{1}] \to [0,1] \) and \( \eta_{2}: [ x_{2}, \tilde x_{2}]\to [0,1] \) by
\[
  \eta_1(x) \coloneqq \frac{x-\tilde x_1}{x_1-\tilde x_1}
  \quand 
  \eta_2(x)\coloneqq \frac{x-x_2}{\tilde x_2- x_2}.
\]
Then we define a \( C^\infty\) map  \( \xi:[0.1]\to [0,1]\) by 
\begin{equation}
  \label{eq:def-of-xi}
  \xi ( x ) \coloneqq
  \begin{cases}
    0 &\text{if } x = 0 \\
    \exp \left\{ 1 - \frac{ 1 }{ 1 - ( x - 1 )^2 } \right\} &\text{if } x \in (0, 1]. 
  \end{cases}
\end{equation}
and let 
\[
  \xi_1(x):= \xi\circ \eta_1(x)
  \quand 
  \xi_2(x):= \xi\circ \eta_2(x)
\]
Notice that \( \xi\) is  monotone increasing, \( \xi ( 0 ) = 0 \), \( \xi ( 1 ) = 1 \),  and   \( D^k \xi ( 0 ) = D^k \xi ( 1 ) = 0 \) for every \( k \geq 1 \) and therefore \( \xi_1, \xi_2\) are also in particular \( C^\infty\) and flat at the endpoints. 
We now fix 
\( \tilde x_1\) to be the mid-point between \( -1\) and \( x_1\), and \( \tilde x_2\) to be the midpoint between \( x_2\) and \( 1 \), i.e. 
\[
  \tilde x_1 \coloneqq \frac{1}{2} x_1 -  \frac{1}{2}, \quand \tilde x_2 \coloneqq \frac{1}{2} x_2 + \frac{1}{2}. 
\]
\begin{rem}
The fundamental reason for these choices is that,  by using the explicit forms of \( \eta_1, \eta_2\) and \( \xi\), and the definition of \( \xi_1, \xi_2\), we can verify by repeated use of the chain rule that there exists a \( C > 0 \) such that for every \( x_1 \in U_{-1} \), every \( x \in [x_1, \tilde x_1] \)  and for every \( k \leq \lceil \ell_1 \rceil \) we have 
\begin{equation}
 \label{eq:derv-xi-1}
 D^k \xi_1 ( x ) = D^{k-1} \left( \frac{2}{ 1 + x_1 } \xi'( \eta_1 (x) ) \right) = \left( \frac{2}{1 + x_1} \right)^{ k } D^k \xi ( \eta_1 (x) ) \leq C (1 + x_{1})^{-k}.
\end{equation}
and, similarly,   for every \( x_2 \in U_{1} \), every \( x \in [x_2, 1] \) and every \( k \leq \lceil \ell_2 \rceil \), we have 
\begin{equation*}
 D^k \xi_2 ( x ) \leq C (1 - x_{2})^{-k}.
\end{equation*}
\end{rem}

Next, we will fix \( \tilde b_1, \tilde b_2 > 0 \). These constants play no role in the statistical properties of the map but will need to be chosen carefully to ensure that \( g \) indeed satisfies \ref{itm:A0}-\ref{itm:A2}. We have already that \( \tilde{b}_1, \tilde{b_2} \) are small enough so that \eqref{eq:cond-monotonicty} holds, and now we also  require the following conditions:
\begin{equation}
  \label{eq:cond-on-b1}
  \begin{aligned}
  &\text{if } \ell_1 > 0, \text{then } &b_1 \ell_1 (1 + x)^{\ell_1} - \tilde b_1 \tilde \ell_1 ( 1 + x )^{\tilde \ell_1} &&\geq 0 
  &&\text{ for all } x \in [\tilde x_1, x_1],\\
  &\text{if } \ell_1 = 0, \text{then } &1 - \tilde b_1 \tilde \ell (1 + x)^{\tilde \ell_1 } &&\geq 0 
  &&\text{ for all } x \in [\tilde x_1, x_1],
  \end{aligned}
\end{equation}
and
\begin{equation}
  \label{eq:cond-on-b2}
  \begin{aligned}
  &\text{if } \ell_2 > 0, \text{then } &b_2 \ell_2 (1 + x)^{\ell_2} - \tilde b_2 \tilde \ell_2 ( 1 + x )^{\tilde \ell_2} &&\geq 0 
  &&\text{ for all } x \in [x_2, \tilde x_2 ],\\
  &\text{if } \ell_2 = 0, \text{then } &1 - \tilde b_1 \tilde \ell_2 (1 + x)^{\tilde \ell_2 } &&\geq 0 
  &&\text{ for all } x \in [x_2, \tilde x_2 ].
  \end{aligned}
\end{equation}
Note that it is always possible to find \( \tilde b_1, \tilde b_2 > 0 \) small enough so that the above hold.

This concludes our definition of the map \( g \) which depends only \( f \) and the parameters \( \tilde \ell_1, \tilde \ell_2 > 0 \) and \( x_1, x_2 \). Notice that choice of \( \tilde \ell_1, \tilde \ell_2 > 0 \) is completely arbitrary and the only restriction on the choice of \( x_1, x_2 \) is that they have to lie in the neighbourhoods \( U_{-1}, U_{1} \), in particular notice that \( x_1,x_2 \) can be chosen arbitrarily close to the fixed points \( -1,1 \).

\subsection{The map \( g \) is in the class \( \widehat{ \mathfrak{F} } \)}
\label{sec:fhat}

\begin{lem}
  \label{lem:fhat}
  \( g \) belongs to the class \( \widehat{ \mathfrak{F} } \) with  parameters \( \tilde{\ell}_1, \tilde{\ell}_2, k_{1}, k_2, a_{1}, a_2, \tilde b_1, \tilde b_2 \).  
\end{lem}

\begin{rem}
  Note that all of the parameters of \( g \) are fixed by the map \( f \) \emph{except} for \( \tilde \ell_1, \tilde \ell_2 \) which are \emph{arbitrary} positive constants, which means that we are free to choose \( \tilde \ell_1 , \tilde \ell_2 \) so that \( g \) is in any one of the three distinct subclasses \( \mathfrak{F}, \mathfrak{F}_{\pm}, \mathfrak{F}_{*} \) defined in \eqref{eq:def-subclasses}.
\end{rem}

\begin{proof}
It follows immediately  from the construction that the map is full branch and \( C^{2} \) and that \ref{itm:A0}  is therefore satisfied. The neighbourhoods \( U_{\pm 1}\) and \( U_{0\pm}\) no longer have the required form as in \ref{itm:A1} but we can shrink them and define the intervals 
  \begin{equation}
    \label{eq:def-tilde-U}
    \begin{split}
        \widetilde{U}_{-1} \coloneqq [-1, \tilde x_{1} ), 
        \quad 
        \widetilde U_{1} \coloneqq ( \tilde x_{2}, 1 ], 
        \quad 
         \widetilde{U}_{0-} = g^{-1} (\widetilde U_{1}),
         \quad 
        \widetilde{U}_{0+} \coloneqq g^{-1} (   \widetilde{U}_{-1}). 
    \end{split}
  \end{equation}
It then follows  that  \( g \)  satisfies \ref{itm:A1} with respect to these neighbourhoods for the required parameters \( \tilde{\ell}_1, \tilde{\ell}_2, k_{1}, k_2, a_{1}, a_2, \tilde b_1, \tilde b_2 \). Since we have shrunk the neighbourhoods and modified the map in the regions \( U_{-1}\setminus   \widetilde{U}_{-1} \) and \( U_{1}\setminus   \widetilde{U}_{1} \) it is no longer immediate that the expansivity condition  \ref{itm:A2} continues to hold outside these new neighbourhoods and we therefore just need to check \ref{itm:A2}.

  Let \( \tilde{\delta}_{n}^{\pm}, \tilde{\Delta}_n^{\pm} \) and \( \delta_n^{\pm}, \Delta_n^{\pm} \) denote the partitions corresponding to \( g \) and \( f \) respectively. We know from the construction of \( g \) that \( g^{n} \equiv f^{n} \) on \( \delta_n^{\pm} \), and \( \tilde \delta_n^{\pm} = \delta_n^{\pm} \) for every \( n \leq n^{\pm} \coloneqq \min \{ x : \delta_n \subset U_{ 0 \pm } (f) \}\). As \( f \) satisfies \ref{itm:A2} we know that
  \begin{equation}
    \label{eq:A2-for-f}
    (g|_{\delta_n^{\pm}}^n)'(x) = (f|_{\delta_n^{\pm}}^n)'(x) > \lambda > 1 \text{ for every } n \leq n^{\pm}.
  \end{equation}
  So, in order to verify that \( g \) satisfies \ref{itm:A2} it remains to check that
  \begin{equation}
    \label{eq:g-A2-main-claim}
    (g^n)' (x) > \lambda > 1, \text{ for every } x \in \tilde \delta_n^{\pm} \text{ and for every }  n^{\pm} + 1 \leq n \leq \tilde{n}^{\pm},
  \end{equation}
  for some (possibly different) \( \lambda > 1  \),
  where \( \tilde{n}^{\pm} \coloneqq \min \{ x : \delta_n \subset \tilde U_{ 0 \pm } (f) \}\).
  We will follow the argument given \cite{CoaLuzMub22}*{Section 3.3} and show that \( (g^{n+1})' (x) \geq (g^{n})'(x) \) for every \( n \geq n^{+} \) and every \( x \in \tilde \delta_n^{+} \). The argument for \( n \geq n^{-} \) and \( x \in \tilde{ \delta }_{n}^- \) then follows in the same way. Note that if \( k_1, \in ( 0, 1 ] \) then there is nothing to prove, so we will assume throughout this proof that \( k_1 > 1 \).
  As in \cite{CoaLuzMub22}*{Section 3.3}  we define
  \begin{equation}\label{eq:def-phi}
      \phi \coloneqq ( g|_{ U_{0+} } )^{-1} \circ g|_{ U_{-1} } \circ g |_{ U_{0+}}
  \end{equation}
  and claim that the conclusion of \cite{CoaLuzMub22}*{Lemma 3.7} holds for \( g \), namely we claim that
  \begin{equation}\label{eq:phiexpansion}
   \frac{(g^2)'(x)}{ g' ( \phi (x) )} = \frac{ g'(x) }{ g'(\phi(x))} g'(g(x)) > 1, \text{ for every }  x \in \tilde \delta_{n+1}^{+}  \text{ and every } n \geq n^{+} .
  \end{equation}

  Notice that if \( g (x) \in [ x_1, 0) \), or if \( g(x) \in [-1, \tilde x_1] \) then we can apply\cite{CoaLuzMub22}*{Lemma 3.7}   to obtain \( (g^2)'(x) / g'( \phi (x)) > 1 \). If instead \( g(x) \in ( \tilde x_1 , x_1 ) \), then \cite{CoaLuzMub22}*{Lemma 3.7} cannot be applied directly, but its proof can be adapted to our setting as show below. So, let us assume that \( x \in \delta_n^+ \) for some \( n \) and that \( g(x) \in ( \tilde x_1 , x_1) \). Since we are working only with \( x \in \delta_n^+ \) we will drop the subscripts on the parameters to ease notation, specifically we will let \( \ell = \ell_1 \), \( \tilde \ell = \tilde \ell_1 \), \(b = b_1 \), \()\tilde b = \tilde b_1\), \(k = k_2\) and \( a = a_2 \). By the definition of \( g \) in \( U_{0+}\) given in \ref{itm:A1}  we have
  \begin{equation}
    \label{eq:g-p-over-g-phi}
    \frac{g'(x)}{g'(\phi(x))} = \left(\frac{ x }{ \phi (x) }\right)^{k - 1} =
    \left(\frac{ \phi (x) }{ x }\right)^{1-k}
  \end{equation}
  Recall that \( k > 1\) and \( x< \phi(x)\)  and so the ratio above is strictly less than \( 1\). Let us fix
  \begin{equation}
    \label{eq:def-y}
    y \coloneqq g ( x ) = -1 + a x^{k},
  \end{equation}
  We will compute \( ( \phi(x) / x )^{ k } \) and \( g' ( g (x) ) \) in case that \( \ell = 0 \) and the case that \( \ell > 0 \) separately. First we note that regardless the value of \( \ell \) we find that inserting the definition of \( g \) into \eqref{eq:def-phi} yields
  \begin{equation}    \label{eq:phi-x}
    \varphi (x) = \left( \frac{1}{a} \right)^{1/k} ( 1 + g(y) )^{ 1/k } 
    = \left( \frac{1}{a} \right)^{1/k} [ 1 + h_1 (y) + \xi (y) ( f(y) - h_1 (y) ) ]^{ 1/k }.
  \end{equation}
  and, using \eqref{eq:cond-monotonicty} we have that 
  \begin{equation}  \label{eq:g-p-g} 
  g'( g(x)  ) = h_1'(y) + \xi( y) ( f'(y) - h_1'(y) ) + \xi'(y) ( f (y) -  h(y) ) \geq h_1'(y) + \xi( y) ( f'(y) - h_1'(y) )
  \end{equation}

  1) Suppose that \( \ell > 0 \).
  Inserting \eqref{eq:def-y} into the definitions of \( f \) and \( h_1 \) we find
  \( f(y) - h_1(y) = a x^k (b a^{\ell} x^{ \ell } - \tilde b a^{\tilde \ell} x^{ \tilde \ell } ) \), and we note for later use that \eqref{eq:cond-monotonicty} ensures that
  \begin{equation}
    \label{eq:d-positive}
    b a^{\ell} x^{ \ell } - \tilde b a^{\tilde \ell} x^{ \tilde \ell } \geq 0.
  \end{equation}
  Thus, from \eqref{eq:phi-x}, 
  \begin{align*}
    \left( \frac{ \phi (x) }{ x } \right)^{k}
    &=\left( \frac{1}{ a x^k } \right) \left[ ax^k + \tilde b a^{\tilde \ell + 1} x^{k ( \tilde \ell + 1 )} + \xi (y) a x^k \left( b a^{\ell} x^{ \ell } - \tilde b a^{\tilde \ell} x^{ \tilde \ell } \right) \right] 
    \\
    &= 1 + \tilde b a^{\tilde \ell} x^{k\tilde \ell} + \xi (y)\left( b a^{\ell} x^{ \ell } - \tilde b a^{\tilde \ell} x^{ \tilde \ell } \right) \\
  \end{align*}
  Next, using \eqref{eq:cond-on-b1}, \eqref{eq:g-p-g} and \eqref{eq:d-positive},   we obtain
  \begin{align}
    \nonumber
    g'( g(x)  )
    \nonumber
    &> 1 + \tilde b x^{ k \tilde \ell } +  \xi( y)( f'(y) - h_1'(y) ) \\
    \nonumber
    &= 1 + \tilde b x^{ k \tilde \ell } + \xi( y )( (1 + \ell ) b a^{\ell} x^{k \ell} - (1 + \tilde \ell )\tilde b a^{\tilde \ell} x^{k \tilde \ell}) \\
        &\geq 1 + \tilde b x^{ k \tilde \ell } + \xi( y ) \left( b a^{\ell} x^{k \ell} -  \tilde b a^{\tilde \ell} x^{k \tilde \ell} \right) 
    = \left( \frac{ \phi (x) }{ x } \right)^{k}.
       \label{eq:g-prime-g}
  \end{align}
  So, from \eqref{eq:g-p-over-g-phi} and \eqref{eq:g-prime-g} we get 
  \[
  \frac{ g'(x) }{ g'( \phi (x) ) } g' ( g (x ) )
  > \left( \frac{ \phi (x) }{ x  } \right)^{k - 1} \left( \frac{ \phi (x) }{ x  } \right)^{k} > 1,
  \]
  Which proves our claim \eqref{eq:phiexpansion}, in the case that \( \ell > 0 \).

  2) If \( \ell = 0  \), then we proceed as before and insert the definitions of \( f \) and \( h_1 \)
  into \eqref{eq:def-phi}
  \begin{align*}
    \phi(x) = \left(\frac{1}{a} \right)^{1/k} \left[ ax^k + \tilde b a^{\ell + 1 } x^{k ( \tilde \ell + 1 )} + \xi(y)\left( - 1 + \eta (y) + b a x^{k} - \tilde b a^{1 + \tilde \ell} x^{ k ( 1 + \tilde \ell )} \right)  \right],
  \end{align*}
  and so 
  \[
    \left( \frac{ \phi (x) }{ x } \right)^{k} = 
    1 + \tilde b a^{ \tilde \ell} x^{k \tilde \ell } + \xi(y)\left( \frac{- 1 + \eta (y)}{ ax^{k}} + b - \tilde b a^{\tilde \ell} x^{ k \tilde \ell } \right).
  \]
  We recall from \cite{CoaLuzMub22}*{Equation(13)} that \( \eta''(y) \implies \eta' (y) \geq \eta (y) / ( 1 +y ) > ( \eta (y) - 1 )/ ( 1 + y ) \) for every \( y \in U_{-1} \).
  So, inserting the expressions for \( f' \) and \( h_1' \) into \eqref{eq:g-p-g}
  \begin{align*}
    g'(y)  &\geq 1 + \tilde b ( 1 + \tilde \ell ) (1 + y )^{\tilde \ell} + \xi (y) \left( \eta'(y) + 1 + b - \tilde b (1 + \tilde \ell ) (1 + y )^{ \tilde \ell } \right) \\
    &> 1 + \tilde b a^{\tilde \ell} x^{ k\tilde \ell} + \xi(y) \left( \frac{-1 + \eta (y) }{ ax^{k} } + b - \tilde b a^{\tilde \ell} x^{ k \tilde \ell } + 1 - \tilde b \tilde\ell a ^{\tilde \ell} x^{ k \tilde \ell} \right)
    \geq \left( \frac{ \phi (x) }{ x } \right)^{k}.
  \end{align*}
  This concludes \eqref{eq:phiexpansion} in the case that \( \ell = 0 \).
  We can then proceed as in \cite{CoaLuzMub22}*{Corollary 3.9}, to get 
  \[
  (g^{n+1}) ' (x) = g'(x) g'(g(x)) \cdots g'(g^{n} (x))
  = \frac{ g'(x) g'(g (x))}{ g'( \phi(x)) }  (g^{n})'(\phi (x))
  > (g^{n})'(\phi (x)).
  \]
  This, together with \eqref{eq:A2-for-f}, implies \eqref{eq:g-A2-main-claim} and allows us to conclude that \( g \) satisfies \ref{itm:A2}.
\end{proof}

\subsection{\( g \) is \( C^r \) close to \( f \)}
\label{sec:crclose}

Let \( f \in \widehat{ \mathfrak{F} } \). In Sections \ref{sec:construction} and \ref{sec:fhat} we  constructed a map  \( g \in \widehat{ \mathfrak{F} } \) ultimately depending only on the choice of two arbitrary constants  \( \tilde \ell_1, \tilde \ell_2 > 0 \) and two points  \( x_1 \in U_{-1} \), \( x_{2} \in U_{1} \). We now show that \( g \) can be chosen arbitrarily close to \( f \), by choosing the points \( x_{1}, x_{2}\) sufficiently close to the fixed points \( -1, 1\) respectively, in a topology determined by the constants 
\begin{equation}\label{eq:approxreg}
  r_{1} \coloneqq \min \{ \ell_1, \tilde \ell_1 \},
  \quand r_{2} \coloneqq \min \{ \ell_1, \tilde \ell_2 \}.
\end{equation}
More precisely,  we have the following result.

\begin{lem}
  \label{lem:c-r-close}
  There exists a \( C > 0 \) such that
  \[
    \| f - g \|_{ C^{ \lceil r_1 \rceil }([-1,0]) } \leq C ( 1 + x_1 )^{ 1 + r_1 - \lceil r_1 \rceil },
    \quand
    \| f - g \|_{ C^{ \lceil r_2 \rceil }([0,1])  } \leq C ( 1 - x_2 )^{ 1 + r_2 - \lceil r_2 \rceil }.
  \]
\end{lem}

\begin{proof}
  We will only give an explicit proof of the bound for \( \| f - g \|_{C^{\lceil r_1 \rceil }} \) as the argument for \( \| f - g \|_{C^{\lceil r_2 \rceil }} \) is the same.
  If \( \ell_1 = \tilde{\ell_1} \) we obtain trivially that \( f - g \equiv 0 \). So, let us assume \( \ell_1 \neq \tilde{\ell_1} \) and consider the subcases \( \ell_1 = 0 \), \( \ell_1 > 0 \) separately.
  If \( \ell_1 = 0 \), then \( r_1 = 0 \) and using the definition \eqref{eq:def-h1} of \( h_1 \) near \( -1 \), the fact that \( \xi \) is bounded, and the definition of \( g \) one finds that
  \[
    \| g - f \|_{C^0} \leq 2 \| h_1 - f \|_{C^0} < C (1 + x) \leq C (1 + x_1).
  \]
  This proves the result in the case that \( \ell_1 = 0. \)
  Let us now suppose throughout the remainder of the proof that \( \ell_1 > 0 \). We begin by establishing the following sublemma.

  \begin{sublem}
    \label{sublem:derv-bounds}
    There exists a \( C > 0 \) so that for every \( 1 \leq k \leq \lceil r_1 \rceil \)
    \begin{equation}
      \label{eq:derv-bounds}
      \begin{split}
      &| D^k f(x) - D^k h_1 (x) | \leq C( 1 +  x_{1})^{1 + r_1 - k }\quad \forall x \in [-1, x_1] \\
      \end{split}
    \end{equation}
  \end{sublem}

  \begin{proof}[Proof of Sublemma \ref{sublem:derv-bounds}]
    Suppose that \( 0 < r_1 = \ell_1 < \tilde \ell_1 \), let \( 1 \leq k \leq \lceil r_1 \rceil \)
    and note that for some constants \( c_k, \tilde c_k \)
    \begin{align*}
      | D^k f (x) - D^k h_1 (x) | &= | c_{k} ( 1 + x )^{ 1 + \ell_1 - k } + \tilde c_{k} ( 1 + x )^{ 1 + \tilde \ell_1 - k } | 
      = ( 1 + x )^{ 1 + \ell_1 - k } | c_k - \tilde c_{k}( 1 + x ) ^{\tilde \ell_1 - \ell_1} |.
    \end{align*}
    As \( ( 1 + x ) ^{\tilde \ell_1 - \ell_1} \to 0 \) as \( x \to -1 \), and as we are considering only finitely many \( k \), we see that there exists some constant \( C > 0 \), independent of \( k \) such that \eqref{eq:derv-bounds} holds.
    Suppose now that \( 0 < \tilde \ell_1 < \ell_1 \). Repeating the calculation above with  \( \ell_1 \) and \( \tilde \ell_1 \) exchanged we conclude the proof.
  \end{proof}

  We now continue the proof of Lemma \ref{lem:c-r-close}. The bound \eqref{eq:derv-bounds} immediately implies 
  \[
   \| f - g \|_{C^{\lceil r_1 \rceil} [-1, \tilde x_1 ]} = \| f - h_1 \|_{C^{\lceil r_1 \rceil} [-1, \tilde x_1 ]} \leq C (1 + x_1)^{1 + r_1 - \lceil r_1 \rceil}
   \]
    for some \( C \) which depends on \( \lceil r_1 \rceil \), but not on \( x_1 \). For \( x \in [\tilde{x}_1, x_1] \) we find by repeated applications of the product rule that
  \begin{equation}\label{eq:derv-of-g}
    D^k g(x)
    = D^k h_1 (x) + \sum_{ j = 0 }^{ n } \binom{k}{j} D^j \xi_1(x) \cdot ( D^{k -j} f (x) - D^{ k - j} h_1 (x) ).
  \end{equation}
  Using \eqref{eq:derv-xi-1} and using \eqref{eq:derv-bounds}, we obtain
  \begin{align}
    \nonumber
    \left|\sum_{ j = 0 }^{ n } \binom{k}{j} D^j \xi_1 \cdot ( D^{k -j} f (x) - D^{ k - j} h_1 (x) ) \right|
    &\leq C \sum_{ j = 0 }^{ n } \binom{k}{j} ( 1 + x_1)^{-j} x_1^{ 1 + r_1 - k + j}  \\
    \label{eq:sum-of-derv}
    &= C ( 1 + x_1)^{1 + r_1 -k}.
  \end{align}
  Finally, combining \eqref{eq:derv-bounds}, \eqref{eq:derv-of-g} and \eqref{eq:sum-of-derv} we find that for any \( k = 0 , \ldots, \lceil r_1 \rceil \) we find that
  \(
  |D^k f (x) - D^k g (x)  | \leq   C( 1 + x_1)^{1 + r_1 - k }
  \), for every \( x \in [-1, -1 + x_1] \).
  So,
  \begin{equation}\label{eq:Cr-norm-g-1}
    \| f - g \|_{ C^{\lceil r_1 \rceil} [ -1, 0 ] } \leq C (1 + x_1)^{1 + r_1  - \lceil r_1 \rceil },
  \end{equation}
  for some \( C \) which does not depend on \( x_1 \).
\end{proof}

\subsection{Concluding the proof of Theorem \ref{thm:density-main}}
\label{sec:conc}

\begin{proof}
  Let \( f \in \widehat{\mathfrak{F}} \) and let \( \varepsilon > 0 \). For each of the classes \( \mathfrak{F},  \mathfrak{F}_{\pm},\mathfrak{F}_* \) we will choose   \( \tilde \ell_1, \tilde \ell_2 > 0 \) so that the corresponding map \( g \) constructed in Section \ref{sec:construction} belongs to the chosen class, and then, by Lemma~\ref{lem:c-r-close}, we can choose  \( x_1 \in U_{-1},  x_2 \in U_{1} \) so that \( g \)  is \( \varepsilon \)-close to \( f\) in the appropriate topology.

We illustrate this process in detail in a couple of  cases and then give some tables to show that choices in all cases. Suppose first that \( f \in \mathfrak{F} \) (so that  \( \beta\in [0,1)\) and  \( f \) has a physical measure  equivalent to Lebesgue) and let us approximate \( f \) by some map \( g\in \mathfrak F_{*}\),  ( so that \( \beta\geq 1\) and \( g \) has  no physical measure). Set \( \tilde \ell_{1}=1/k_{2},  \tilde \ell_{2}=1/k_{1}\) so that \( \tilde\ell_{1}k_{2}= \tilde\ell_{2}k_{1}=1\), which ensures that \( g\in \mathfrak F_{*}\). Notice that we could  choose  \( \tilde \ell_{1}=t/k_{2},  \tilde \ell_{2}=t/k_{1}\) for any \( t \geq 1 \)  and that for any such choice we have  \( \tilde\ell_{1} > \ell_{1}\) and \( \tilde\ell_{2}> \ell_{2}\) because by assumption we have \( \ell_{1}k_{2}, \ell_{2}k_{1}\in [0,1)\). 
Once \( \tilde\ell_{1}, \tilde\ell_{2}\) have been chosen we immediately get the regularity of the approximation from \eqref{eq:approxreg} which in this case is given by \( r= \lceil \min\{\ell_{1}, \tilde\ell_{1}, \ell_{2}, \tilde\ell_{2}\}\rceil =  \lceil\min\{\ell_{1},  \ell_{2}\}\rceil= r_{*}(f)  \), thus proving the Theorem in this case.

For a second example, suppose that \( f \in \mathfrak{F}_{\pm} \) with \( 0 < \beta^+ < 1 \leq \beta^- \) (so that \( f \) has a physical measure supported on the fixed point \( -1 \)), and let us construct a \( \tilde g \in \mathfrak{F} \) (so that \( g \) has a physical measure equivalent to Lebesgue)  that is close to \( f \). Recall that \( \tilde g \in \mathfrak{F} \) if and only if \(  \beta^+, \beta^- \in [0,1) \)
and so we  can leave unchanged the value of \( \beta^{+}\), and therefore let \( \tilde\ell_{2}=\ell_{2}\), but we need to lower the value of \( \beta^{-}\) to something less than 1. We can do this by letting  \( \tilde\ell_{1} = 1/ k_2 - \gamma \)  for any \( 0 < \gamma < 1/ k_2 \), which  gives \( \beta^{-}(g)=\tilde\ell_{2}k_{2}= ( 1/ k_2 - \gamma) k_{2} = 1- k_{2}\gamma <1\), and therefore   \( g \in \mathfrak{F} \) as required.
To estimate the distance between  \( f \) and \( g \), and to choose the appropriate metric for this distance, notice that \( g|_{[0,1]}=f|_{[0,1]}\) and therefore we only need to worry about the distance between \( f\) and \( g \) on \( [-1,0]\). Therefore,  by \eqref{eq:approxreg}
we get \( r = \lceil \min\{\ell_{1}, \tilde \ell_{1}\} \rceil  =\lceil \ell_{1} \rceil  =\lceil 1 / k_2 - \gamma \rceil \) and by Lemma \ref{lem:c-r-close} we can construct \( g \) arbitrarily close to \( f \) in the \( d_{r}\) metric with \( r =  \lceil 1 / k_2 - \gamma \rceil \) for any \( \gamma>0\) arbitrarily small. Notice however that if \( \gamma\) is sufficiently small then  \( \lceil 1 / k_2 - \gamma \rceil = \lceil 1 / k_2 \rceil \) and therefore \( r = \tilde r(f) \) as claimed in the Theorem.

All the cases can be obtained by a simple reasoning as illustrated in the two examples above, from which we deduce the choices for \( \tilde\ell_{1}, \tilde\ell_{2}\) and the corresponding regularity of the approximation given in the Table 1 below, thus completing the proof. 
  \renewcommand{\arraystretch}{1.4}
  \begin{table}[h]
  \label{table1}
    \centering
    \subfloat[Choice of \( \tilde{\ell}_1, \tilde{\ell}_2 \) for constructing \( \tilde{f} \)]{
    \begin{tabular}{@{}lcc@{}}
      \toprule
      Parameters of \( f \) & \( \tilde{ \ell }_1 \) & \( \tilde{ \ell }_2 \)  \\
      \midrule
      \( 1 \leq \beta^+ < 1 \leq \beta^- \) & \( 1/k_2 - \gamma \)  & \( \ell_2 \)   \\
      \( 1  \leq \beta^- < 1 \leq \beta^+ \) & \(  \ell_1 \) & \(  1/k_1 - \gamma \)  \\
      \( \beta^+, \beta^-  \geq 1 \) & \(1/k_2 - \gamma \) &  \( 1/k_1 - \gamma \)  \\
      \bottomrule
    \end{tabular}}
    \\
    \quad
    \subfloat[Choice of \( \tilde{\ell}_1, \tilde{\ell}_2 \) for constructing \( f_{\pm} \)]{
    \begin{tabular}{@{}lcc@{}}
      \toprule
      Parameters of \( f \) & \( \tilde{ \ell }_1 \) & \( \tilde{ \ell }_2 \)  \\
      \midrule
      \( \ell_1 \geq \ell_2 \) and \( \beta^- < 1 \)  & \(  1/k_2 \) & \(  \ell_2 \)   \\
      \( \ell_1 \geq \ell_2 \) and \( \beta^- \geq 1 \)  & \(  \ell_1 + \gamma \) & \(  \ell_2 \) \\
      \( \ell_1 < \ell_2 \) and \( \beta^+ < 1 \)  & \(  \ell_1 \) & \(  1 / k_1 \)   \\
      \( \ell_1 < \ell_2 \) and \( \beta^+ \geq 1 \)  & \(  \ell_1 \) & \(  \ell_2 + \gamma \)   \\
      \bottomrule
    \end{tabular}
    }
    \\
    \vspace{1em}
    \subfloat[Choice of \( \tilde{\ell}_1, \tilde{\ell}_2 \) for constructing \( f_* \)]{
      \begin{tabular}{@{}lcc@{}}
        \toprule
        Parameters of \( f \) & \( \tilde{ \ell }_1 \) & \( \tilde{ \ell }_2 \)  \\
        \midrule
        \( \beta \in [0,1) \) & \(  1/k_2 \) & \(  1/k_1 \)   \\
        \( \beta^+ < 1 \leq \beta^- \) & \(  \ell_1  \) & \(  \beta^- / k_1 \)   \\
        \( \beta^- < 1 \leq \beta^+  \) & \(  \beta^+ / k_2  \) & \(  \ell_2  \)   \\
        \( \beta^+, \beta^- \geq 1 \) and \( k_1 \geq k_2 \)  & \(  \beta^+ / k_2  \) & \(  \ell_2 \)   \\
        \( \beta^+, \beta^- \geq 1 \) and \( k_1 < k_2 \)  & \(  \ell_1  \) & \(  \beta^-/ k_1 \)   \\
        \bottomrule
    \end{tabular}
    }
    \caption{Choosing \(\tilde{\ell}_1, \tilde{\ell}_2 \) to complete the proof of Theorem \ref{thm:density-main}.}
  \end{table}
\end{proof}

\newpage



\begin{bibdiv}
\begin{biblist}


  \bib{AarDen2001}{article}{
    author = {Aaronson, Jon and Denker, Manfred},
    date-added = {2023-06-14 14:31:37 -0300},
    date-modified = {2023-06-14 14:31:37 -0300},
    doi = {10.1142/S0219493701000114},
    fjournal = {Stochastics and Dynamics},
    issn = {0219-4937},
    journal = {Stoch. Dyn.},
    mrclass = {37A50 (28D05 37A05 60F05)},
    mrnumber = {1840194},
    mrreviewer = {Dieter H. Mayer},
    number = {2},
    pages = {193--237},
    title = {Local limit theorems for partial sums of stationary sequences generated by {G}ibbs-{M}arkov maps},
    url = {https://doi-org.uoelibrary.idm.oclc.org/10.1142/S0219493701000114},
    volume = {1},
    year = {2001}
  }

\bib{AarThaZwe05}{article}{
      author={Aaronson, Jon},
      author={Thaler, Maximilian},
      author={Zweim{\"u}ller, Roland},
       title={Occupation times of sets of infinite measure for ergodic
  transformations},
        date={2005},
     journal={Ergodic Theory and Dynamical Systems},
}

\bib{Alv20}{book}{
      author={Alves, Jos{\'e}~F.},
       title={Nonuniformly hyperbolic attractors. geometric and probabilistic
  aspects},
      series={Springer Monographs in Mathematics},
   publisher={Springer International Publishing},
        date={2020},
}

\bib{AlvDiaLuz17}{article}{
      author={Alves, Jos{\'e}~F.},
      author={Dias, Carla~L.},
      author={Luzzatto, Stefano},
      author={Pinheiro, Vilton},
       title={S{RB} measures for partially hyperbolic systems whose central
  direction is weakly expanding},
        date={2017},
     journal={J. Eur. Math. Soc. (JEMS)},
      volume={19},
      number={10},
       pages={2911\ndash 2946},
}

\bib{AlvLuzPin05}{article}{
      author={Alves, Jos{\'e}~F.},
      author={Luzzatto, Stefano},
      author={Pinheiro, Vilton},
       title={Markov structures and decay of correlations for non-uniformly
  expanding dynamical systems},
        date={2005},
     journal={Ann. Inst. H. Poincar\'e Anal. Non Lin\'eaire},
      volume={22},
      number={6},
       pages={817\ndash 839},
}

\bib{AnoSin67}{article}{
      author={Anosov, D~V},
      author={Sinai, Yakov~G.},
       title={{Some Smooth Ergodic Systems}},
        date={1967},
     journal={Russian Mathematical Surveys},
      volume={103},
}

\bib{AraLuzVia09}{article}{
      author={Ara{\'u}jo, V{\'{\i}}tor},
      author={Luzzatto, Stefano},
      author={Viana, Marcelo},
       title={Invariant measures for interval maps with critical points and
  singularities},
        date={2009},
     journal={Adv. Math.},
      volume={221},
      number={5},
       pages={1428\ndash 1444},
}

\bib{AraPin21}{article}{
      author={Ara{\'{u}}jo, V{\'{i}}tor},
      author={Pinheiro, Vilton},
       title={{Abundance of wild historic behavior}},
        date={2021},
     journal={Bulletin of the Brazilian Mathematical Society. New Series.},
      volume={52},
}

\bib{Bac99}{article}{
      author={Bachurin, P~S},
       title={The connection between time averages and minimal attractors},
        date={1999},
     journal={Russian Mathematical Surveys},
      volume={54},
      number={6},
       pages={1233\ndash 1235},
}

\bib{BahGalNis18}{article}{
      author={Bahsoun, Wael},
      author={Galatolo, Stefano},
      author={Nisoli, Isaia},
      author={Niu, Xiaolong},
       title={A rigorous computational approach to linear response},
        date={2018},
     journal={Nonlinearity},
      volume={31},
}

\bib{BahSau16}{article}{
      author={Bahsoun, Wael},
      author={Saussol, Beno{\^\i}t},
       title={Linear response in the intermittent family: differentiation in a
  weighted $c^0$-norm},
        date={2016},
     journal={Discrete and Continuous Dynamical Systems},
}

\bib{BalTod16}{article}{
      author={Baladi, V.},
      author={Todd, M.},
       title={Linear response for intermittent maps},
        date={2016},
     journal={Communications in Mathematical Physics},
      volume={347},
       pages={857\ndash 874},
}

\bib{BarKirNak20}{article}{
      author={Barrientos, Pablo~G.},
      author={Kiriki, Shin},
      author={Nakano, Yushi},
      author={Raibekas, Artem},
      author={Soma, Teruhiko},
       title={Historic behavior in non-hyperbolic homoclinic classes},
        date={2020},
     journal={Proceedings of the American Mathematical Society},
      volume={148},
       pages={1195\ndash 1206},
}

\bib{BerBie22}{article}{
      author={Berger, Pierre},
      author={Biebler, Sebastien},
       title={Emergence of wandering stable components},
        date={2022},
     journal={Journal of the American Mathematical Society},
      volume={36},
}

\bib{Bir31}{article}{
      author={Birkhoff, George~David},
       title={{Proof of the Ergodic Theorem.}},
        date={1931},
     journal={Proceedings of the National Academy of Sciences of the United
  States of America},
      volume={17},
      number={12},
       pages={656\ndash 660},
}

\bib{Bow75}{book}{
      author={Bowen, Rufus},
       title={Equilibrium states and the ergodic theory of {A}nosov
  diffeomorphisms},
      series={Lecture Notes in Mathematics, Vol. 470},
   publisher={Springer-Verlag},
     address={Berlin},
        date={1975},
}

\bib{BruLep13}{article}{
      author={Bruin, Henk},
      author={Leplaideur, Renaud},
       title={Renormalization, thermodynamic formalism and quasi-crystals in
  subshifts},
        date={2013feb},
     journal={Communications in Mathematical Physics},
      volume={321},
      number={1},
       pages={209\ndash 247},
}

\bib{BruTerTod19}{article}{
      author={Bruin, Henk},
      author={Terhesiu, Dalia},
      author={Todd, Mike},
       title={The pressure function for infinite equilibrium measures},
        date={2019jun},
     journal={Israel Journal of Mathematics},
      volume={232},
      number={2},
       pages={775\ndash 826},
}

\bib{Bur21}{article}{
      author={Burguet, David},
       title={{SRB} measures for ${C}^\infty$ surface diffeomorphisms},
        date={2021},
     journal={Preprint},
}

\bib{Buz00}{article}{
      author={Buzzi, J{\'e}r{\^o}me},
       title={Absolutely continuous invariant probability measures for
  arbitrary expanding piecewise r-analytic mappings of the plane},
        date={2000},
     journal={Ergodic Theory Dynam. Systems},
      volume={20},
      number={3},
       pages={697\ndash 708},
}

\bib{BuzCroSar22}{article}{
      author={Buzzi, J{\'e}r{\^o}me},
      author={Crovisier, Sylvain},
      author={Sarig, Omri},
       title={Another proof of burguet's existence theorem for srb measures of
  $c^\infty$ surface diffeomorphisms},
        date={2022},
     journal={Preprint},
}

\bib{CamIso95}{article}{
      author={Campanino, Massimo},
      author={Isola, Stefano},
       title={Statistical properties of long return times in type i
  intermittency},
        date={1995},
     journal={Forum Mathematicum},
      volume={7},
      number={7},
}

\bib{CliLuzPes17}{article}{
      author={Climenhaga, Vaughn},
      author={Luzzatto, Stefano},
      author={Pesin, Yakov},
       title={The geometric approach for constructing {S}inai-{R}uelle-{B}owen
  measures},
        date={2017},
     journal={Journal of Statistical Physics},
      volume={166},
}

\bib{CliLuzPes23}{article}{
      author={Climenhaga, Vaughn},
      author={Luzzatto, Stefano},
      author={Pesin, Yakov},
       title={Srb measures and young towers for surface diffeomorphisms},
        date={2023},
     journal={Annales Henri Poincar{\'{e}}},
      volume={23},
}

\bib{CoaHolTer19}{article}{
      author={Coates, Douglas},
      author={Holland, Mark},
      author={Terhesiu, Dalia},
       title={Limit theorems for wobbly interval intermittent maps},
        date={201910},
}

\bib{CoaLuzMub22}{article}{
      author={Coates, Douglas},
      author={Luzzatto, Stefano},
      author={Muhammad, Mubarak},
       title={Doubly intermittent full branch maps with critical points and
  singularities},
        date={2023/09/01},
     journal={Communications in Mathematical Physics},
      volume={402},
      number={2},
       pages={1845\ndash 1878},
         url={https://doi.org/10.1007/s00220-023-04766-x},
}


\bib{ColVar01}{article}{
      author={Colli, Eduardo},
      author={Vargas, Edson},
       title={Non-trivial wandering domains and homoclinic bifurcations},
        date={2001},
     journal={Ergodic Theory and Dynamical Systems},
}

\bib{CriHayMar10}{article}{
      author={Cristadoro, Giampaolo},
      author={Haydn, Nicolai},
      author={Marie, Philippe},
      author={Vaienti, Sandro},
       title={Statistical properties of intermittent maps with unbounded
  derivative},
        date={2010},
     journal={Nonlinearity},
}

\bib{CroYanZha20}{article}{
      author={Crovisier, Sylvain},
      author={Yang, Dawei},
      author={Zhang, Jinhua},
       title={Empirical measures of partially hyperbolic attractors},
        date={2020},
     journal={Communications in Mathematical Physics},
}

\bib{Cui21}{article}{
      author={Cui, Hongfei},
       title={Invariant densities for intermittent maps with critical points},
        date={2021},
     journal={Journal of Difference Equations and Applications},
      volume={27},
      number={3},
       pages={404\ndash 421},
}

\bib{DiaHolLuz06}{article}{
      author={Diaz-Ordaz, Karla},
      author={Holland, Mark},
      author={Luzzatto, Stefano},
       title={Statistical properties of one-dimensional maps with critical
  points and singularities},
        date={2006},
     journal={Stochastics and Dynamics},
      volume={6},
      number={4},
}

\bib{Dol04}{incollection}{
    author = {Dolgopyat, Dmitry},
     title = {Prelude to a kiss},
 booktitle = {Modern dynamical systems and applications},
     pages = {313--324},
      year = {2004},
}

\bib{Dua12}{article}{
      author={Duan, Y.},
       title={A.c.i.m for random intermittent maps: existence, uniqueness and
  stochastic stability},
        date={2012},
     journal={Dynamical Systems. An International Journal},
}

\bib{FisLop01}{article}{
      author={Fisher, Albert~M},
      author={Lopes, Artur},
       title={Exact bounds for the polynomial decay of correlation, 1/fnoise
  and the {CLT} for the equilibrium state of a non-h{\"o}lder potential},
        date={2001jul},
     journal={Nonlinearity},
      volume={14},
      number={5},
       pages={1071\ndash 1104},
}

\bib{FreFreTod13}{article}{
      author={Freitas, Ana Cristina~Moreira},
      author={Freitas, Jorge~Milhazes},
      author={Todd, Mike},
       title={The compound poisson limit ruling periodic extreme behaviour of
  non-uniformly hyperbolic dynamics},
        date={2013mar},
     journal={Communications in Mathematical Physics},
      volume={321},
      number={2},
       pages={483\ndash 527},
}

\bib{FreFreTod16}{article}{
      author={Freitas, Ana Cristina~Moreira},
      author={Freitas, Jorge~Milhazes},
      author={Todd, Mike},
      author={Vaienti, Sandro},
       title={Rare events for the manneville-pomeau map},
        date={2016nov},
     journal={Stochastic Processes and their Applications},
      volume={126},
      number={11},
       pages={3463\ndash 3479},
}

\bib{FroMurSta11}{article}{
      author={Froyland, Gary},
      author={Murray, Rua},
      author={Stancevic, Ognjen},
       title={Spectral degeneracy and escape dynamics for intermittent maps
  with a hole},
        date={2011jul},
     journal={Nonlinearity},
      volume={24},
      number={9},
       pages={2435\ndash 2463},
}

\bib{GalHolPer21}{article}{
      author={Galatolo, Stefano},
      author={Holland, Mark},
      author={Persson, Tomas},
      author={Zhang, Yiwei},
       title={Anomalous time-scaling of extreme events in infinite systems and
  birkhoff sums of infinite observables},
        date={202110},
     journal={Discrete and Continuous Dynamical Systems},
}

  \bib{Gou2010}{article}{
    author = {Gou\"{e}zel, S\'{e}bastien},
    date-added = {2023-08-31 14:52:40 -0300},
    date-modified = {2023-08-31 14:52:40 -0300},
    doi = {10.1007/s11856-010-0092-z},
    fjournal = {Israel Journal of Mathematics},
    issn = {0021-2172},
    journal = {Israel J. Math.},
    mrclass = {37A50 (28D05 37A30 60F05)},
    mrnumber = {2735054},
    mrreviewer = {Michael M. Bj\"{o}rklund},
    pages = {1--41},
    title = {Characterization of weak convergence of {B}irkhoff sums for {G}ibbs-{M}arkov maps},
    url = {https://doi.org/10.1007/s11856-010-0092-z},
    volume = {180},
    year = {2010}
  }
  
\bib{Gou10a}{article}{
      author={Gou{\"{e}}zel, S{\'{e}}bastien},
       title={{Almost sure invariance principle for dynamical systems by
  spectral methods}},
        date={2010jul},
     journal={The Annals of Probability},
      volume={38},
      number={4},
       pages={1639\ndash 1671},
}

\bib{Her18}{incollection}{
      author={Herman, Michel},
       title={An example of non-convergence of birkhoff sums},
        date={2018},
   booktitle={Notes inachev{\'e}es de michael r. herman s{\'e}lectionn{\'e}es
  par jean-christophe yoccoz},
   publisher={Soci{\'e}t{\'e} Math{\'e}matique de France},
}

\bib{HofKel90}{article}{
      author={Hofbauer, Franz},
      author={Keller, Gerhard},
       title={Quadratic maps without asymptotic measure},
        date={1990},
     journal={Comm. Math. Phys.},
      volume={127},
      number={2},
       pages={319\ndash 337},
}

\bib{HofKel95}{incollection}{
      author={Hofbauer, Franz},
      author={Keller, Gerhard},
       title={{Quadratic maps with maximal oscillation}},
        date={1995},
   booktitle={Algorithms, fractals, and dynamics},
      editor={Takahashi, Y.},
   publisher={Springer, Boston, MA},
       pages={89\ndash 94},
}

\bib{HuYou95}{article}{
      author={Hu, Huyi},
      author={Young, Lai-Sang},
       title={Nonexistence of sbr measures for some diffeomorphisms that are
  `almost anosov'},
        date={1995},
     journal={Ergodic Theory and Dynamical Systems},
      volume={15},
}

\bib{Ino00}{article}{
      author={Inoue, Tomoki},
       title={Sojourn times in small neighborhoods of indifferent fixed points
  of one-dimensional dynamical systems},
        date={2000},
     journal={Ergodic Theory and Dynamical Systems},
      volume={20},
}

\bib{JarTol04}{book}{
      author={J{\"a}rvenp{\"a}{\"a}, Esa},
      author={Tolonen, Tapani},
       title={Natural ergodic measures are not always observable},
   publisher={University of Jyv{\"a}skyl{\"a}},
        date={2005},
}

\bib{KanKirLi16}{article}{
      author={Kanagawa, Hiratuka},
      author={Kiriki, Shin},
      author={Li, Ming-Chia},
       title={Geometric {L}orenz flows with historic behaviour},
        date={2016},
     journal={Discrete and Continuous Dynamical Systems},
      volume={36},
      number={12},
       pages={7021\ndash 7028},
}

\bib{Kel04}{article}{
      author={Keller, Gerhard},
       title={{Completely mixing maps without limit measure}},
        date={2004},
     journal={Colloquium Mathematicum},
      volume={100},
      number={1},
       pages={73\ndash 76},
}

\bib{KirLiSom10}{article}{
      author={Kiriki, Shin},
      author={Li, Ming-Chia},
      author={Soma, Teruhiko},
       title={{Coexistence of invariant sets with and without SRB measures in
  Henon family}},
        date={2010},
     journal={Nonlinearity},
      volume={23},
      number={9},
       pages={2253\ndash 2269},
}

\bib{KirLiNak22}{article}{
      author={Kiriki, Shin},
      author={Li, Xiaolong},
      author={Nakano, Yushi},
      author={Soma, Teruhiko},
       title={Abundance of observable lyapunov irregular sets},
        date={2022},
     journal={Communications in Mathematical},
      volume={Physics},
       pages={1\ndash 29},
}

\bib{KirNakSom19}{article}{
      author={Kiriki, Shin},
      author={Nakano, Yushi},
      author={Soma, Teruhiko},
       title={Historic behaviour for nonautonomous contraction mappings},
        date={2019},
     journal={Nonlinearity},
      volume={32},
       pages={1111\ndash 1124},
}

\bib{KirNakSom21}{article}{
      author={Kiriki, Shin},
      author={Nakano, Yushi},
      author={Soma, Teruhiko},
       title={Historic and physical wandering domains for wild
  blender-horseshoes},
        date={2021},
     journal={Preprint},
}

\bib{KirNakSom22}{article}{
      author={Kiriki, Shin},
      author={Nakano, Yushi},
      author={Soma, Teruhiko},
       title={Emergence via non-existence of averages},
        date={2022},
     journal={Advances in Mathematics},
      volume={400},
       pages={1\ndash 30},
}

\bib{KirSom17}{article}{
      author={Kiriki, Shin},
      author={Soma, Teruhiko},
       title={Takens' last problem and existence of non-trivial wandering
  domains},
        date={2017},
     journal={Advances in Mathematics,},
      volume={306},
       pages={pp.},
}

\bib{Kle06}{article}{
      author={Kleptsyn, V~A},
       title={An example of non-coincidence of minimal and statistical
  attractors},
        date={2006},
     journal={Ergodic Theory and Dynamical Systems},
      volume={26},
}

\bib{Kor16}{article}{
      author={Korepanov, Alexey},
       title={Linear response for intermittent maps with summable and
  nonsummable decay of correlations},
        date={2016},
     journal={Nonlinearity},
}

\bib{LabRod17}{article}{
      author={Labouriau, Isabel~S.},
      author={Rodrigues, Alexandre A.~P.},
       title={On takens' last problem: tangencies and time averages near
  heteroclinic networks},
        date={2017},
     journal={Nonlinearity},
      volume={30},
       pages={1876\ndash 1910},
}

\bib{LasYor73}{article}{
      author={Lasota, A.},
      author={Yorke, James~A.},
       title={On the existence of invariant measures for piecewise monotonic
  transformations},
        date={1973},
     journal={Trans. Amer. Math. Soc.},
      volume={186},
       pages={481\ndash 488},
}

\bib{LivSauVai99}{article}{
      author={Liverani, Carlangelo},
      author={Saussol, Beno{\^{\i}}t},
      author={Vaienti, Sandro},
       title={A probabilistic approach to intermittency},
        date={1999},
     journal={Ergodic Theory Dynam. Systems},
      volume={19},
      number={3},
       pages={671\ndash 685},
}

\bib{Mel08}{article}{
      author={Melbourne, Ian},
       title={Large and moderate deviations for slowly mixing dynamical
  systems},
        date={2008nov},
     journal={Proceedings of the American Mathematical Society},
      volume={137},
      number={5},
       pages={1735\ndash 1741},
}

\bib{MelTer12}{article}{
      author={Melbourne, Ian},
      author={Terhesiu, Dalia},
       title={First and higher order uniform dual ergodic theorems for
  dynamical systems with infinite measure},
        date={2012nov},
     journal={Israel Journal of Mathematics},
      volume={194},
      number={2},
       pages={793\ndash 830},
}

\bib{NicTorVai16}{article}{
      author={Nicol, Matthew},
      author={T{\=o}r{\=o}k, Andrew},
      author={Vaienti, Sandro},
       title={Central limit theorems for sequential and random intermittent
  dynamical systems},
        date={2016},
     journal={Ergodic Theory and Dynamical Systems},
      volume={38},
      number={3},
       pages={1127\ndash 1153},
}

\bib{Nol20}{book}{
      author={Nolan, John~P},
       title={Univariate stable distributions: models for heavy tailed data},
   publisher={Springer},
        date={2020},
}

\bib{KarAsh11}{article}{
      author={{\=o}zkan Karabacak},
      author={Ashwin, Peter},
       title={On statistical attractors and the convergence of time averages},
        date={2011},
     journal={Mathematical Proceedings of the Cambridge Philosophical Society},
      volume={150},
}

\bib{Pal08}{article}{
      author={Palis, J},
       title={Open questions leading to a global perspective in dynamics},
        date={2008},
     journal={Nonlinearity},
      volume={21},
      number={4},
}

\bib{Pal15}{article}{
      author={Palis, J},
       title={Open questions leading to a global perspective in dynamics
  (corrigendum)},
        date={2015},
     journal={Nonlinearity},
      volume={28},
      number={3},
}

\bib{Pet2002}{article}{
  author={Petrov, Valentin~V.},
  title={A note on the {B}orel-{C}antelli lemma},
  date={2002},
  ISSN={0167-7152},
  journal={Statist. Probab. Lett.},
  volume={58},
  number={3},
  pages={283\ndash 286},
  url={https://doi.org/10.1016/S0167-7152(02)00113-X},
  review={\MR{1921874}},
}

\bib{Pia80}{article}{
      author={Pianigiani, Giulio},
       title={First return map and invariant measures},
        date={1980},
     journal={Israel J. Math.},
      volume={35},
      number={1-2},
       pages={32\ndash 48},
}

\bib{Pin06}{article}{
      author={Pinheiro, Vilton},
       title={Sinai-{R}uelle-{B}owen measures for weakly expanding maps},
        date={2006},
     journal={Nonlinearity},
      volume={19},
      number={5},
       pages={1185\ndash 1200},
}

\bib{PolSha09}{article}{
      author={Pollicott, Mark},
      author={Sharp, Richard},
       title={Large deviations for intermittent maps},
        date={2009jul},
     journal={Nonlinearity},
      volume={22},
      number={9},
       pages={2079\ndash 2092},
}

\bib{PolWei99}{article}{
      author={Pollicott, Mark},
      author={Weiss, Howard},
       title={Multifractal analysis of lyapunov exponent for continued fraction
  and manneville-pomeau transformations and applications to diophantine
  approximation},
        date={1999nov},
     journal={Communications in Mathematical Physics},
      volume={207},
      number={1},
       pages={145\ndash 171},
}

\bib{PomMan80}{article}{
      author={Pomeau, Yves},
      author={Manneville, Paul},
       title={Intermittent transition to turbulence in dissipative dynamical
  systems},
        date={1980},
     journal={Communications in Mathematical Physics},
      volume={74},
       pages={189\ndash 197},
}

\bib{Rue76}{article}{
      author={Ruelle, David},
       title={A measure associated with axiom-{A} attractors},
        date={1976},
     journal={Amer. J. Math.},
      volume={98},
      number={3},
       pages={619\ndash 654},
}

\bib{Ruz15}{article}{
      author={Ruziboev, Marks},
       title={Decay of correlations for invertible systems with non-h{\"o}lder
  observables},
        date={2015},
     journal={Dynamical Systems. An International Journal},
}

\bib{Ruz18}{incollection}{
      author={Ruziboev, Marks},
       title={Almost sure rates of mixing for random intermittent maps},
        date={201801},
   booktitle={Differential equations and dynamical systems},
   publisher={Springer},
}

\bib{Sar01}{article}{
      author={Sarig, Omri},
       title={{Phase transitions for countable Markov shifts}},
        date={2001},
     journal={Communications in Mathematical Physics},
      volume={217},
      number={3},
       pages={555\ndash 577},
}

\bib{Saw66}{article}{
      author={Sawyer, S.},
       title={Maximal inequalities of weak type},
        date={1966},
     journal={Annals of Mathematics},
      volume={84},
}

\bib{SheStr13}{article}{
      author={Shen, Weixiao},
      author={van Strien, Sebastian},
       title={On stochastic stability of expanding circle maps with neutral
  fixed points},
        date={2013sep},
     journal={Dynamical Systems},
      volume={28},
      number={3},
       pages={423\ndash 452},
}

\bib{Sin72}{article}{
      author={Sinai, Ja.~G.},
       title={Gibbs measures in ergodic theory},
        date={1972},
     journal={Uspehi Mat. Nauk},
      volume={27},
      number={4},
       pages={21\ndash 64},
}

\bib{Tak94}{article}{
      author={Takens, Floris},
       title={Heteroclinic attractors: time averages and moduli of topological
  conjugacy.},
        date={1994},
     journal={Bullettin of the Brazilian Mathematical Society},
      volume={25},
}

\bib{Tak08}{article}{
      author={Takens, Floris},
       title={Orbits with historic behaviour, or nonexistence of averages},
        date={2008},
     journal={Nonlinearity},
      volume={21},
}

\bib{Tal20}{article}{
      author={Talebi, Amin},
       title={Statistical (in)stability and non-statistical dynamics},
        date={2020},
     journal={Preprint},
}

\bib{Tal22}{article}{
      author={Talebi, Amin},
       title={Non-statistical rational maps},
        date={2022},
     journal={Mathematische Zeitschrift},
}

\bib{Ter13}{article}{
      author={Terhesiu, Dalia},
       title={Improved mixing rates for infinite measure-preserving systems},
        date={2013aug},
     journal={Ergodic Theory and Dynamical Systems},
      volume={35},
      number={2},
       pages={585\ndash 614},
}

\bib{Ter15}{article}{
      author={Terhesiu, Dalia},
       title={Mixing rates for intermittent maps of high exponent},
        date={2015},
     journal={Probability Theory and Related Fields},
      volume={166},
      number={3-4},
       pages={1025\ndash 1060},
}

\bib{Tha80}{article}{
      author={Thaler, Maximilian},
       title={Estimates of the invariant densities of endomorphisms with
  indifferent fixed points},
        date={1980},
     journal={Israel Journal of Mathematics},
      volume={37},
      number={4},
       pages={303\ndash 314},
}

\bib{Tha83}{article}{
      author={Thaler, Maximilian},
       title={Transformations on [0, 1] with infinite invariant measures},
        date={1983},
     journal={Israel Journal of Mathematics},
      volume={46},
      number={1-2},
       pages={67\ndash 96},
}

\bib{Tha95}{article}{
      author={Thaler, Maximilian},
       title={The invariant densities for maps modeling intermittency},
        date={1995},
     journal={Journal of Statistical Physics},
      volume={79},
      number={3-4},
       pages={739\ndash 741},
}

\bib{Tha95a}{article}{
      author={Thaler, Maximilian},
       title={A limit theorem for the perron-frobenius operator of
  transformations on [0,1] with indifferent fixed points},
        date={1995},
     journal={Israel Journal of Mathematics},
      volume={91},
      number={1-3},
       pages={111\ndash 127},
}

\bib{Tha00}{article}{
      author={Thaler, Maximilian},
       title={The asymptotics of the perron-frobenius operator of a class of
  interval maps preserving infinite measures},
        date={2000},
     journal={Studia Mathematica},
      volume={143},
      number={2},
       pages={103\ndash 119},
}

\bib{Tha05}{article}{
      author={Thaler, Maximilian},
       title={Asymptotic distributions and large deviations for iterated maps
  with an indifferent fixed point},
        date={2005},
     journal={Stochastics and Dynamics},
      volume={05},
      number={03},
       pages={425\ndash 440},
}

\bib{Tsu00c}{article}{
      author={Tsujii, Masato},
       title={Absolutely continuous invariant measures for piecewise
  real-analytic expanding maps on the plane},
        date={2000},
     journal={Comm. Math. Phys.},
      volume={208},
      number={3},
       pages={605\ndash 622},
}

\bib{Tsu01a}{article}{
      author={Tsujii, Masato},
       title={Absolutely continuous invariant measures for expanding piecewise
  linear maps},
        date={2001},
     journal={Invent. Math.},
      volume={143},
      number={2},
       pages={349\ndash 373},
}

\bib{Tsu05}{article}{
      author={Tsujii, Masato},
       title={Physical measures for partially hyperbolic surface
  endomorphisms},
        date={2005},
     journal={Acta Math.},
      volume={194},
      number={1},
       pages={37\ndash 132},
}

\bib{Vec22}{article}{
      author={Veconi, Dominic},
       title={SRB measures of singular hyperbolic attractors},
        date={2022},
     journal={Discrete and Continuous Dynamical Systems},
      volume={42},
}

\bib{You99}{article}{
      author={Young, Lai-Sang},
       title={Recurrence times and rates of mixing},
        date={1999},
     journal={Israel J. Math.},
      volume={110},
       pages={153\ndash 188},
}

\bib{Zwe98}{article}{
      author={Zweim{\"{u}}ller, Roland},
       title={{Ergodic structure and invariant densities of non-Markovian
  interval maps with indifferent fixed points}},
        date={1998},
     journal={Nonlinearity},
      volume={1263},
}

\bib{Zwe00}{article}{
      author={Zweim{\"{u}}ller, Roland},
       title={{Ergodic properties of infinite measure-preserving interval maps
  with indifferent fixed points}},
        date={2000},
     journal={Ergodic Theory and Dynamical Systems},
      volume={20},
      number={5},
       pages={1519\ndash 1549},
}

\bib{Zwe02}{article}{
      author={Zweim{\"{u}}ller, Roland},
       title={{Exact $C^\infty$ covering maps of the circle without (weak)
  limit measure}},
        date={2002},
     journal={Colloquium Mathematicum},
      volume={93},
      number={2},
       pages={295\ndash 302},
}

\bib{Zwe03}{article}{
      author={Zweim{\"{u}}ller, Roland},
       title={{Stable limits for probability preserving maps with indifferent
  fixed points}},
        date={2003},
     journal={Stochastics and Dynamics},
      volume={3},
      number={1},
       pages={83\ndash 99},
}



\end{biblist}
\end{bibdiv}
\end{document}